\documentclass{amsart}
\usepackage{amsthm,amsmath,amssymb, color, graphics, enumerate}
\usepackage{mathrsfs}
\usepackage[all]{xy}
\usepackage{mathtools}

\begin{document}

\theoremstyle{plain}
\newtheorem*{nclaim}{Claim}
\newtheorem{mythm}{Theorem}[]
\newtheorem{thm}{Theorem}[section]
\newtheorem{lem}[thm]{Lemma}
\newtheorem{cor}[thm]{Corollary}
\newtheorem{note}[thm]{Note}
\newtheorem{prop}[thm]{Proposition}
\newtheorem{fact}[thm]{Fact}
\newtheorem{claim}[thm]{Claim}
\newtheorem{subclaim}[thm]{Sublaim}
\newtheorem{conj}[thm]{Conjecture}
\newtheorem{obs}[thm]{Observation}
\newtheorem{q}[thm]{Question}
\newtheorem{mlem}[thm]{Main Lemma}
\newtheorem{notation}[thm]{Notation}
\newtheorem{defnthm}[thm]{Definition and Theorem}
\newtheorem{recall}[thm]{Recall}
\newtheorem*{nthm}{Theorem}
\newtheorem*{nq}{Question}
\newtheorem*{nmc}{Main Claim}

\theoremstyle{definition}
\newtheorem{defn}[thm]{Definition}

\theoremstyle{remark}
\newtheorem{remark}[thm]{Remark}
\newtheorem{example}[thm]{Example}

\def\RR{\mathbb{R}}
\def\ZZ{\mathbb{Z}}
\def\cN{\mathcal{N}}
\def\cM{\mathcal{M}}
\def\cK{\mathcal{K}}
\def\cE{\mathcal{E}}
\def\cI{\mathcal{I}}
\def\cJ{\mathcal{J}}
\def\cA{\mathcal{A}}
\def\cX{\mathcal{X}}
\def\cY{\mathcal{Y}}
\def\cS{\mathcal{S}}
\def\cB{\mathcal{B}}
\def\cC{\mathcal{C}}
\def\cQ{\mathcal{Q}}
\def\cL{\mathcal{L}}
\def\add#1{{\operatorname{\textup{\textsf{add}}}}(#1)}
\def\non#1{{\operatorname{\textup{\textsf{non}}}}(#1)}
\def\cov#1{{\operatorname{\textup{\textsf{cov}}}}(#1)}
\def\cof#1{{\operatorname{\textup{\textsf{cof}}}}(#1)}
\def\adds#1{{\operatorname{\textup{\textsf{add}}}}^*(#1)}
\def\nons#1{{\operatorname{\textup{\textsf{non}}}}^*(#1)}
\def\covs#1{{\operatorname{\textup{\textsf{cov}}}}^*(#1)}
\def\cofs#1{{\operatorname{\textup{\textsf{cof}}}}^*(#1)}
\def\fc{\mathfrak{c}}
\def\fb{\mathfrak{b}}
\def\fd{\mathfrak{d}}
\def\fe{\mathfrak{e}}
\def\fg{\mathfrak{g}}
\def\fh{\mathfrak{h}}
\def\fk{\mathfrak{k}}
\def\fm{\mathfrak{m}}
\def\fp{\mathfrak{p}}
\def\fr{\mathfrak{r}}
\def\fs{\mathfrak{s}}
\def\ft{\mathfrak{t}}
\def\fw{\mathfrak{w}}
\def\fN{\mathfrak{N}}
\def\fM{\mathfrak{M}}
\def\fY{\mathfrak{Y}}
\def\fT{\mathfrak{T}}
\def\fX{\mathfrak{X}}
\def\isomorphic{\simeq}
\def\cf{\operatorname{cf}}
\def\Cof{\operatorname{Cof}}
\def\op#1#2{\langle #1,#2\rangle}
\def\dom{\operatorname{dom}}
\def\ran{\operatorname{ran}}
\def\card#1{\lvert#1\rvert}
\def\size#1{\left|#1\right|}
\def\reals#1{{#1}^{#1}}
\def\set#1{\left\{{#1}\right\}}
\def\seq#1{\left\langle{#1}\right\rangle}
\def\forallbutfin{\forall^{\infty}}
\def\existinf{\exists^{\infty}}
\newcommand{\forces}[1][{}]{\Vdash_{#1}}
\newcommand{\forcestext}[2][{}]
{\Vdash_{#1}\text{``}\,{#2}\,\text{''}}
\newcommand{\tvalue}[1][{}]{{[\![#1]\!]}}
\def\V{\mathbf{V}}
\def\W{\mathbf{W}}
\def\On{\mathbf{On}}
%
%{\Large $<$}$\!\!$\raisebox{2pt}{\footnotesize $\circ$}
\def\emb{\text{\Large $<$}$\!\!$\raisebox{1.5pt}
{\footnotesize $\circ$}}
\newcommand{\modelstext}[1][{}]
{\models\text{``}\,{#1}\,\text{''}}
\def\that{\!\!\!\text{``}\,\,}
\def\taht{\,\text{''}}
\def\ZF{\operatorname{\sf ZF}}
\def\ZFC{\operatorname{\sf ZFC}}
\def\CH{\operatorname{\sf CH}}
\def\GCH{\operatorname{\sf GCH}}
\def\AC{\operatorname{\sf AC}}
\def\AD{\operatorname{\sf AD}}
\def\MA{\operatorname{\sf MA}}
\def\PFA{\operatorname{\sf PFA}}
\def\OCA{\operatorname{\sf OCA}}
\def\SH{\operatorname{\sf SH}}
\def\PID{\operatorname{\sf PID}}
\def\cD{\mathcal{D}}
\def\cF{\mathcal{F}}
\def\cG{\mathcal{G}}
\def\cH{\mathcal{H}}
\def\cP{\mathcal{P}}
\def\cT{\mathcal{T}}
\def\AA{\mathbb{A}}
\def\AAS{\mathbb{AS}}
\def\BB{\mathbb{B}}
\def\CC{\mathbb{C}}
\def\DD{\mathbb{D}}
\def\EE{\mathbb{E}}
\def\FF{\mathbb{F}}
\def\GG{\mathbb{G}}
\def\II{\mathbb{I}}
\def\LL{\mathbb{L}}
\def\MM{\mathbb{M}}
\def\MMI{\mathbb{MI}}
\def\NN{\mathbb{N}}
\def\OO{\mathbb{O}}
\def\PP{\mathbb{P}}
\def\QQ{\mathbb{Q}}
\def\SS{\mathbb{S}}
\def\TT{\mathbb{T}}
\def\XX{\mathbb{X}}
\def\cSN{\mathcal{SN}}
\def\ind#1{{\operatorname{\rm ind}}(#1)}
\def\cO{\mathcal{O}}
\def\cU{\mathcal{U}}
\def\sfX{\textup{\textsf{X}}}
\def\d{\operatorname{\rm d}}
\def\sfP{\mathsf{P}}
\def\sfY{\mathsf{Y}}
\def\sfD{\mathsf{D}}
\def\sfE{\mathsf{E}}
\def\sfC{\mathsf{C}}
\def\sfF{\mathsf{F}}
\def\sfg{\mathsf{g}}
\def\sfa{\mathsf{a}}
\def\sfb{\mathsf{b}}
\def\sfc{\mathsf{c}}
\def\sfh{\mathsf{h}}
\def\sfi{\mathsf{i}}
\def\sfl{\mathsf{l}}
\def\Fn{\operatorname{{\rm Fn}}}
\def\sfsi{\textup{\textsf{si}}}
\def\sfSI{\textup{\textsf{SI}}}
\def\Gap{\textup{\textsf{Gap}}}
\def\stem{\textup{\textsf{stem}}}
\def\Split{\textup{\textsf{Split}}}
\def\split
{\operatorname{\textup{\textsf{split}}}}
\def\Succ{\textup{\textsf{Succ}}}\def\split
{\operatorname{\textup{\textsf{split}}}}
\def\succ
{\operatorname{\textup{\textsf{succ}}}}
\def\fsc{\mathfrak{sc}}
\def\conc{\,{}^\frown}
\def\blackbox{\rule{8pt}{8pt}}
\def\pend{\dashv}
\def\ot{\rm ot}
\def\wcondition{{\bf1}}
\def\Kunen{{\sf Kunen}}
\def\Laver{{\sf Laver}}
\def\FK{{\sf Freezing}_{\sf Kunen}}
\def\FW{{\sf Freezing}_{\sf Woodin}}
\def\Pregap{\textup{\textsf{Pregap}}}
\def\Lim{\textup{\textsf{Lim}}}
\def\supp{\mathrm{supp}}
\def\loc{\mathbb{LOC}}

%%%restrictedto%%%
\newcommand{\sprestT}{\upharpoonright}
\newcommand{\sprestS}{\upharpoonright}
\newcommand{\restrictedto}{
{\mathchoice{\sprestT}{\sprestT}
{\sprestS}{\sprestS}}}
%%%restrictedto%%%

\def\sprest{\!\downarrow\!}
\def\lAngle{\langle\!\langle} % nath.sty
\def\rAngle{\rangle\!\rangle} % nath.sty

\def\lv{\mathsf{lv}}
\def\height{\mathsf{ht}}
\def\rec{\mathrm{rec}}
\def\arec{\mathrm{arec}}
\def\pK{\text{\rm property K}}
\def\aR{{\mathrm a}\mathsf{R}_{1,\aleph_1}}
\def\R{\mathsf{R}_{1,\aleph_1}}
\def\FSCO{\mathsf{FSCO}}
\def\TA{\mathsf{TA}}
\def\supp{\mathrm{supp}}

\def\bfsc{\sc}
\def\itsc{\sc}
\def\rmsc{\sc}

\renewcommand{\thefootnote}{{\rm $*$\arabic{footnote}}}

\pagestyle{plain}

\title[]{Some infinitely generated non-projective modules over path algebras and their extensions under Martin's Axiom}

\author{Ayako Itaba}
\address{Department of Mathematics, 
Tokyo University of Science, 
1-3 Kagurazaka, Shinjuku-ku, Tokyo 162-8601, JAPAN.}
\email{itaba@rs.tus.ac.jp}

\author{Diego A. Mej\'ia}
\address{Faculty of Science, Shizuoka University, 
Ohya 836, Shizuoka, 422-8529, JAPAN.}
\email{diego.mejia@shizuoka.ac.jp}

\author{Teruyuki Yorioka}
\address{Faculty of Science, Shizuoka University, 
Ohya 836, Shizuoka, 422-8529, JAPAN.}
\email{yorioka@shizuoka.ac.jp}

%\thanks{%Supported by 
%%Grant-in-Aid for 
%%%Young Scientists (B) 25800089,
%%Japan Society for the Promotion of Science.
%%\\
%The authors thank Hiroyuki Minamoto, Izuru Mori 
%and Kenta Ueyama
%for useful comments about this research.
%Especially, they provided us advice and information about 
%{\bfsc Theorem \ref{fdA}},
%{\itsc Remark} {\rmsc \ref{Noether ring}}
%and
%{\bfsc Proposition \ref{closed quiver}}.
%}


\keywords{path algebras, quiver representations,
non-projective modules,
Martin's Axiom
%\\
%\hfill 
%{\tt file name: {\jobname}.tex}
}

\begin{abstract}
In this paper
it is proved that,
when $Q$ is a quiver that admits some closure,
for any algebraically closed field $K$
and 
any finite dimensional $K$-linear representation $\cX$ of $Q$,
if ${\rm Ext}^1_{KQ}(\cX,KQ)=0$
then $\cX$ is projective
({\bfsc Theorem \ref{fqa cor}}).
In contrast,
we show that
if $Q$ is a specific quiver of the type above,
then there is an infinitely generated non-projective $KQ$-module 
$M_{\omega_1}$
such that,
when $K$ is a countable field,
$\MA_{\aleph_1}$ 
(Martin's Axiom for $\aleph_1$ many dense sets,
which is a combinatorial axiom
in set theory)
implies that 
${\rm Ext}^1_{KQ}(M_{\omega_1},KQ)=0$
({\bfsc Theorem \ref{cor final}}).
\end{abstract}

\maketitle

\section*{Introduction}

Bound quiver algebras of finite connected quivers
strongly influence research on representation theory of Artin algebras.
Gabriel found 
a correspondence between 
finite dimensional algebras
and
linear representations of bound quivers 
(\cite{G}, \cite[II]{ASS}),
so
it follows that studying modules of finite dimensional algebras 
is reduced to studying modules of bound quiver algebras. 
In this paper,
we concentrate 
on the study of path algebras,
which is one type of bound quiver algebras.

Nakayama Conjecture, Tachikawa Conjecture,
and Auslander-Reiten Conjecture
are some major research projects
in ring theory that 
present sufficient conditions
for projective modules.
Related to this, 
it has been known the following result
for Artin algebras:
$(*)$
{\em For any finite dimensional algebra\footnote{Any finite dimensional algebra is Artin.}
 $\Lambda$
over an algebraically closed field 
of finite global dimension
and
any finitely generated $\Lambda$-module $M$,
if 
${\rm Ext}^{\geq 1}_\Lambda(M,\Lambda)=0$,
then
$M$ is projective}
({\bfsc Theorem \ref{fdA}}).
A typical example of 
finite dimensional 
%Artin 
algebras
is a path algebra
of a finite acyclic quiver
over an algebraically closed field.
Since
any path algebra 
of a quiver over an algebraically closed field 
is hereditary 
(even when the quiver  is not finite, 
see e.g. \cite[\S8.2]{GR}), 
that is,
its global dimension is not larger than $1$,
the following assertion also holds:
{\em 
For any algebraically closed field $K$,
any finite acyclic quiver $Q$
and 
any finitely generated $KQ$-module $M$,
if 
${\rm Ext}^{1}_{KQ}(M,KQ)=0$,
then
$M$ is projective.}
In {\bfsc Theorem \ref{fqa cor}},
it is shown that 
the above assertion is also true
for finite dimensional $K$-linear representations of 
{\em some} infinite quivers,
one of which is
the following quiver of $A_\infty$ type,
denoted by $A_\infty^{\leftarrow}$:
\[
\xymatrix@C15pt{
0
& 1 \ar[l] 
& 2 \ar[l] 
& \cdots \ar[l]
& n \ar[l]
& n+1 \ar[l]
& \cdots \ar[l]
}
.
\]

Let $A_\infty^{\rightarrow}$ be the opposite quiver 
of $A_\infty^{\leftarrow}$: 
the set of the vertices of $A_\infty^{\rightarrow}$ 
is the same as $A_\infty^{\leftarrow}$,
but the arrows are reversed, that is, 
each arrow in $A_\infty^{\rightarrow}$ is of the form
$\xymatrix@C15pt{n \ar[r] & n+1}$.
The category ${\rm Mod} K A_\infty^{\rightarrow}$ of 
$K A_\infty^{\rightarrow}$-modules 
%has been studied in 
%\cite{Brune:leftright}.
%This category 
is known to be somewhat simple, actually
is pure semisimple, that is, 
every $K A_\infty^{\rightarrow}$-module is a direct sum
of finitely presented $K A_\infty^{\rightarrow}$-modules
\cite[\S2]{Brune:leftright}.
In \cite[Theorem 3.1]{Enochsetc:projrep}, 
a characterization of 
projective representations of $A_\infty^{\rightarrow}$
over a unital ring
is given.
The category of representations of 
$A_\infty^{\leftarrow}$ also has been studied, 
for example, 
a characterization of 
projective representations of $A_\infty^{\leftarrow}$
over a field is presented in 
\cite[p102 {\sc Example}]{Benson:Book},
and 
this is extended to
such representations over a unital ring
in \cite[\S3]{Enochsetc:flatcovers}.

In this paper,
we consider some specific quivers $Q$,
as specified in {\bfsc Theorem \ref{cor final}},
one of which is the quiver $A_\infty^{\leftarrow}$
to
construct an infinitely generated 
{\em non-projective} $KQ$-module, 
which is denoted by
$M_{\omega_1}$.
To analyze such a $KQ$-module $M_{\omega_1}$,
$\MA_{\aleph_1}$ 
(Martin's Axiom for $\aleph_1$ many dense sets)
is used.
$\MA_{\aleph_1}$ is 
a combinatorial axiom of set theory
that cannot neither be proved nor refuted from 
Zermelo-Fraenkel axiomatic set theory {\sf ZFC} 
with the axiom of choice
\cite{MartinSolovay, SolovayTennenbaum: iteration}.
$\MA_{\aleph_1}$ 
is applied in many areas of mathematics
to show that some mathematical statements cannot be refuted 
from {\sf ZFC}
(see e.g. \cite{Fremlin}).
One of such examples is Shelah's solution of
Whitehead Problem
\cite{Shelah:W}.
Our main result states that
{\em
if $K$ is a countable field 
and
$\MA_{\aleph_1}$ holds,
then 
${\rm Ext}^1_{KQ}(M_{\omega_1},KQ)=0$}
({\bfsc Theorems \ref{inf simple thm}},
{\bfsc \ref{2 thm}}
%{\bfsc \ref{fin ext}}
and 
{\bfsc Theorem \ref{cor final}}).
Therefore, 
{\em
under $\MA_{\aleph_1}$
and the assumption that 
$K$ is a countable field,
the above assertion $(*)$ fails for quivers $Q$ 
as in {\bfsc Theorem \ref{cor final}} 
and 
infinitely generated $KQ$-modules}.
Trlifaj's construction is used 
to build such infinitely generated $KQ$-modules,
which will be presented in \S\ref{ods}.

This paper is intended to be fairly self contained, 
but we will assume some basic knowledge about ordinals
(see e.g. \cite[II.1, II.4]{EM} and \cite[I.7, III.6]{Kunen:new}).
\S\ref{prel} provides necessary knowledge,
which includes some facts on path algebras
and 
set theory.
\S \ref{inf gen pa} provides the proof of 
the main result of this paper.

%%%%%%%%%%%%%%%%%%%%
\section{Preliminaries}
\label{prel}

Throughout this paper,
a ring $R$ means a ring with enough idempotents
(hence $R$ may not be unital),
and
an $R$-module means right $R$-module.
For a ring $R$, 
${\rm Mod}R$ denotes 
the category of the $R$-modules,
and
${\rm mod}R$ denotes the category of
the finitely generated $R$-modules.
For an $R$-module $M$ and a subset $X$ of $M$,
$\seq{X}_R$ denotes the $R$-submodule of the module $M$
generated by $X$.
For an $R$-module $M$ and $R$-submodules $N_i$, $i\in I$, of $M$,
$\displaystyle \sum_{i\in I} N_i$
denotes the $R$-submodule that is the $R$-linear span of
the set $\displaystyle \bigcup_{i\in I} N_i$.

We follow the notation of outer direct sums
in \cite[I.2.]{EM}.
For a family $\set{M_i:i\in I}$ of modules,
the {\rm product module}
$\displaystyle \prod_{i\in I} M_i$ is the module
whose underling set is the set of
functions $f$ with domain $I$ such that
for each $i\in I$, 
$f(i)$ belongs to the set $M_i$,
and 
the operations are defined coordinate-wise.
For a member $f$ of the product $\displaystyle \prod_{i\in I} M_i$,
the {\em support} ${\rm supp}(f)$ of $f$ is defined by the set
\[
\set{i\in I: f(i)\neq 0_{M_i}}
.
\]
The outer direct sum
$\displaystyle \bigoplus_{i\in I} M_i$
of a family $\set{M_i:i\in I}$ of modules
is the submodule of the product
module $\displaystyle \prod_{i\in I} M_i$ 
which
consists of the members of the set
$\displaystyle \prod_{i\in I} M_i$ 
whose supports are finite.

We adopt ordinals as the von Neumann ordinals,
that is, 
an ordinal $\alpha$ 
means the set of ordinals less than $\alpha$.
So for ordinals $\alpha$ and $\beta$,
$\alpha$ is less than $\beta$ iff $\alpha\in\beta$.
$\omega$ is the set of all finite ordinals
(non-negative integers),
$\omega_1$ is the least uncountable ordinal 
(which is a cardinal).
$\Lim$ denotes the class of all limit ordinals.

The following is a well-known equivalence about projectivity.

\begin{thm}[\text{\rm E.g. \cite[Propositions 17.1., 17.2]{AF}}]
For a ring $R$ (with enough idempotents)
and an $R$-module $P$,
the following statements are equivalent.
\begin{enumerate}[{\rm (1)}]
\item
For every $R$-epimorphism $f$ from an $R$-module $M$ 
onto an $R$-module $N$ and 
$R$-homomorphism $g$ from $P$ into $N$, 
there exists an $R$-homomorphism $h$ from $P$ into $M$
such that
$g=f\circ h$.

\item
Every $R$-epimorphism from an $R$-module onto $P$
splits, 
that is, it is right invertible.

\item
The functor ${\rm Hom}_R(P, -)$ within the category ${\rm Mod}R$
is exact,
that is,
for every $R$-module $M$,
${\rm Ext}^1_R (P,M)=0$.

\item
$P$ is isomorphic to a direct summand of a free $R$-module.

\end{enumerate}
\end{thm}

%It is known that
%the statements from (1) to (3) are also equivalent
%even if the ring $R$ has no identity.
%\marginpar{Remove the sentence here.}

%%%%%%%%%%%%%%%%%%%
\subsection{Path algebras and quiver representations}
%\marginpar{Change the title of this subsection}

%\begin{thm}[Bahlekeh-Kakaei-Salarian
%\text{\cite[Theorem 3.5]{BKS}}]
%\label{BKS thm}
%For any finite quiver $Q$ and 
%any commutative Noetherian ring $R$,
%if ${\sf ARC}_{R}({\rm mod}R)$ holds,
%then
%${\sf ARC}_{RQ}({\rm mod}RQ)$ holds.
%\end{thm}

%Auslander-Reiten proved that
%${\sf ARC}_{A}({\rm mod}A)$ holds
%for any algebra of finite representation type
%\cite{AR}.
%So in particular, 
%every algebraically closed field $K$ satisfies
%${\sf ARC}_{K}({\rm mod}K)$.

This subsection is devoted to 
the basics of representation theory of rings.
The readers can skip this subsection if they are familiar with 
path algebras and quiver representations.
Quivers, path algebras,
and linear representations of quivers are
some basic concepts of representation theory of Artin algebras.
Our notation and terminology are fairly standard,
see e.g. 
\cite{ARS, ASS}.
In the next paragraphs,
we refer to \cite[Chapters II-III]{ASS}
for definitions, notation,
and terminology.

A {\em quiver} denotes a directed graph.
Any quiver $Q$ consists of a pair of 
a set $Q_0$ of vertices and a set $Q_1$ of arrows.
Each arrow $a$ is equipped with its source $s(a)$ and 
its target $t(a)$.
A quiver $Q=(Q_0,Q_1)$ is called {\em finite}
if both $Q_0$ and $Q_1$ are finite sets.
A {\em path} of the quiver $Q$ is a finite sequence
$a_0 a_1 \cdots a_n$
of arrows of the quiver $Q$
such that,
for each $i$ with $0\leq i <n$,
the target of the arrow $a_i$ coincides with 
the source of the arrow $a_{i+1}$.
The path $a_0 a_1 \cdots a_n$ has length $n+1$.
For each vertex $v$ of the quiver $Q$,
we agree to associate with 
it
a path of length $0$,
called the {\em trivial path} or the {\em stationary path} {\em at 
the vertex} $v$, which is denoted by $e_v$.
A {\em cycle} is a non-trivial path whose source and target coincide.
A quiver is called {\em acyclic} if
there are no cycles in the quiver.
For a quiver $Q$, 
$\overline Q$ denotes the underlying graph of $Q$
that is obtained from $Q$ 
by forgetting the orientation of the arrows,
and 
a quiver $Q$ is called {\em connected} if
the graph $\overline Q$ is a connected graph.
For a field $K$ and a quiver $Q$,
the {\em path algebra $KQ$ of the quiver $Q$ over the field $K$}
is the $K$-algebra
whose underlying set is the $K$-vector space
whose basis is the set of all the paths of the quiver $Q$
(which includes all the stationary paths)
such that
the product of two paths 
$a_0 a_1 \cdots a_{m-1}$ and $b_0 b_1 \cdots b_{n-1}$
is defined as follows:
\[
a_0 a_1 \cdots a_{m-1} \cdot b_0 b_1 \cdots b_{n-1}
=
\left\{
\begin{array}{ll}
a_0 a_1 \cdots a_{m-1} b_0 b_1 \cdots b_{n-1}
& \text{if $t(a_{m-1}) = s(b_0)$}
\\[1em]
0_{KQ} & \text{otherwise.}
\end{array}
\right.
\]
The product of basic elements is extended to
arbitrary elements of $KQ$ by
distributivity.
We note that,
for any field $K$ and a quiver $Q$ with $Q_0$ finite,
$KQ$ also has an identity, 
which is of the form
$\displaystyle
\sum_{v\in Q_0} e_v$.
However, for any quiver $Q$ with infinitely many vertices,
$KQ$ does not have an identity.
We recall that
any path algebra $KQ$
of a quiver $Q$ over an algebraically closed field $K$
is hereditary
even when a quiver $Q$ is not finite
(see e.g. \cite[\S8.2]{GR}), 
that is,
its global dimension is not larger than $1$.

For a quiver $Q=(Q_0,Q_1)$ and a field $K$,
a {\em $K$-linear representation} of the quiver $Q$
is a system $\cX=\seq{\cX_v,\cX_a:v\in Q_0, a\in Q_1}$
such that,
for each vertex $v\in Q_0$,
$\cX_v$ is a $K$-vector space
and, for each arrow $a\in Q_1$,
%with source $v$ and target $w$,
$\cX_a$ is a $K$-linear map
from the $K$-vector space $\cX_{s(a)}$ into 
the $K$-vector space $\cX_{t(a)}$.
A $K$-linear representation is called {\em finite dimensional}
if 
each $\cX_v$, $v\in Q_0$, is a finite dimensional $K$-vector space.
For two $K$-linear representations $\cX$ and $\cY$,
%a {\em morphism from $\cY$ into $\cX$}
a {\em morphism from $\cX$ into $\cY$}
is a tuple $\varphi=\seq{\varphi_v:v\in Q_0}$
such that,
for each $v\in Q_0$,
$\varphi_v$ is a $K$-linear map from
the $K$-vector space $\cX_v$ into 
the $K$-vector space $\cY_v$
and,
for each arrow $a\in Q_1$, 
%with source $v$ and  target $w$,
the following diagram commutes:
%\[
%\xymatrix{
%\cY_{s(a)} \ar[r]^-{\cY_a} \ar[d]_-{\varphi_{s(a)}}
%& \cY_{t(a)} \ar[d]^-{\varphi_{t(a)}}
%\\
%\cX_{s(a)} \ar[r]^-{\cX_a} 
%& \cX_{t(a)}
%}
%\]
\[
\xymatrix{
\cX_{s(a)} \ar[r]^-{\cX_a} \ar[d]_-{\varphi_{s(a)}}
& \cX_{t(a)} \ar[d]^-{\varphi_{t(a)}}
\\
\cY_{s(a)} \ar[r]^-{\cY_a} 
& \cY_{t(a)} 
}
\text{\raisebox{-4em}{.}}
\]
${\rm Rep}_K Q$ denotes
the category of the $K$-linear representations of 
a quiver $Q$ over a field $K$,
and 
${\rm rep}_K Q$ denotes
the category of the finite dimensional 
$K$-linear representations $\cX$ of 
$Q$ over $K$.
In \cite{ASS},
these are defined for finite quivers,
however, 
we adopt them for all quivers.

There is a correspondence between 
$KQ$-modules and $K$-linear representation of $Q$
(see e.g. \cite[Theorem III.1.6]{ASS}).
For a $KQ$-module $M$,
define the $K$-linear representation $F(M)$ of $Q$
such that,
for each $v\in Q_0$,
$
F(M)_v:= M e_v =\set{m e_v: m\in M}
,
$
and,
for each $a\in Q_1$,
$F(M)_a$ is the $K$-homomorphism from $F(M)_{s(a)}$
into $F(M)_{t(a)}$ such that,
for each $x\in F(M)_{s(a)}$, 
$F(M)_a (x)=xa$.
For a $K$-linear representation $\cX$ of $Q$,
define the $KQ$-module $G(\cX)$
whose underlying set is
the direct sum 
$\displaystyle
\bigoplus_{v\in Q_0} \cX_v$
such that,
for each element 
$m=
\displaystyle
\sum_{v\in Q_0} x_v$
of 
$\displaystyle
\bigoplus_{v\in Q_0} \cX_v$
(in this notation,
for all but finitely many $v\in Q_0$,
$x_v$ is the zero of $\cX_v$),
$w\in Q_0$ and $a\in Q_1$,
$
m e_w := x_w
$
and
$
m a
:= \cX_a(m e_{s(a)})
,
$
and 
the product by any arbitrary element of $KQ$ 
is extended by distributivity.
We notice that,
for every $K$-linear representation $\cX$ of $Q$,
$F(G(\cX))=\cX$, 
and, 
for every $KQ$-module $M$,
if $M= \displaystyle \sum_{m \in M} mKQ$
then $G(F(M))=M$.
Therefore,
if $Q$ is a finite connected quiver,
then 
the category ${\rm Mod}KQ$ is equivalent to 
the category ${\rm Rep}_K Q$
by the functors $F$ and $G$
\cite[Theorem III.1.6]{ASS}
and,
for any finite acyclic quiver $Q$,
${\rm mod}KQ$ is equivalent to
${\rm rep}_K  Q$
\cite[Theorem III.1.7]{ASS}.
%When the quiver $Q$ is not finite,
%any object of ${\rm mod}KQ$ corresponds to 
%an object of ${\rm rep}_K Q$ as above,
%but
%the converse is not true in general.

%%%%%%%%%%%%%%%%%%%
\subsection{Path algebras of infinite quivers}
%\marginpar{Insert subsection here.}

%We have learned the following from 
%Hiroyuki Minamoto.
%

Throuout this subsection,
$F$ denotes the canonical functor from 
${\rm Mod}KQ$
to 
${\rm Rep}_K Q$,
and 
$G$ denotes the canonical functor from 
${\rm Rep}_K Q$ to ${\rm Mod} KQ$,
for a field $K$ and a quiver $Q$, 
as in the last paragraph of the previous subsection.
The following theorem gives a sufficient condition
for finitely generated projective modules over 
a finite dimensional 
%Artin 
algebra.
For example,
the following is mentioned without proof 
in the proof of \cite[Theorem 4.7]{Minamoto}.

\begin{thm}[Folklore]
\label{fdA}
Suppose that
$\Lambda$ is a 
finite dimensional 
%Artin 
algebra
over an algebraically closed field $K$
%Noetherian ring
with finite global dimension.
Then
for any finitely generated $\Lambda$-module $M$,
if 
${\rm Ext}^{\geq 1}_\Lambda(M,\Lambda)=0$
then
$M$ is projective.
\end{thm}

\begin{proof}
Suppose that
$M$ is a finitely generated $\Lambda$-module
and
${\rm Ext}^{\geq 1}_\Lambda(M,\Lambda)=0$.
The point of the proof is to show that,
{\em
for any finitely generated projective $\Lambda$-module $P$,
${\rm Ext}^{\geq 1}_\Lambda(M, P)=0$.}
To see this,
let $P'$ be a complementary direct summand of $P$
such that
$P \oplus P' $
is isomorphic to 
a direct sum $\Lambda^{n}$ of finitely many copies
of $\Lambda$.
Then 
for each integer $k\geq 1$,
\[
{\rm Ext}^{k}_{\Lambda}(M,\Lambda^{n})
=\left({\rm Ext}^k_{\Lambda}(M,\Lambda)\right)^n
=0
,
\]
and
\[
{\rm Ext}^k_\Lambda(M,\Lambda^{n})
=
{\rm Ext}^k_\Lambda(M,P)
\oplus
{\rm Ext}^k_\Lambda(M,P')
.
\]
Therefore,
${\rm Ext}^k_\Lambda(M,P)=0$.
Hence
${\rm Ext}^{\geq 1}_\Lambda(M,P)=0$.

Let $d$ be the projective dimension 
${\rm pd}M$ of $M$. 
Since $\Lambda$ has finite global dimension,
$0\leq d < \infty$.
Assume, towards a contradiction,
that $d\geq 1$.
Let 
the sequence 
\[
\xymatrix@C=8pt@R=3pt{
0 \ar[rr] 
&& P_d \ar[rr]^-{f_d} 
&& P_{d-1} \ar[rr]^-{f_{d-1}} \ar[rd]
&& \cdots & \cdots 
&  \cdots \ar[rr]^{f_1} %\ar[rd]
&& P_0 \ar[rr]^{f_0} 
&& M \ar[rr]%^-{f_{-1}} 
&& 0
\\
&&&&& \Omega^{d-1}M \ar[rd] %\ar[ru]
%&&&& K_1 \ar[ru] \ar[rd]
\\
&&&& %0 \ar[ru]
&& 0 
%&& 0 \ar[ru]
%&& 0
}
\]
be a projective resolution of $M$
of length $d$
such that
each $P_i$ is finitely generated,
where
$\Omega^{d-1}M$ is the $(d-1)$-th syzygy of $M$.
Since ${\rm pd }M=d$,
the projective dimension of
the $\Lambda$-module $\Omega^{d-1}M$
is exactly $1$,
in particular, 
$\Omega^{d-1}M$ is not projective.
Applying
${\rm Hom}_\Lambda( - , P_d)$
to the short exact sequence
$0 \to P_d \to P_{d-1} \to \Omega^{d-1}M \to 0$,
we obtain the following exact sequence
\begin{multline*}
\xymatrix@C=10pt{
0 \ar[r] &
{\rm Hom}_\Lambda(\Omega^{d-1}M, P_d) \ar[r] &
}
\\
\xymatrix@C=10pt{
{\rm Hom}_\Lambda(P_{d-1}, P_d) \ar[rrrr]^-{{\rm Hom}_\Lambda(f_d,P_d)}
&&&&
{\rm Hom}_\Lambda(P_{d}, P_d) \ar[r] &
{\rm Ext}^1_\Lambda(\Omega^{d-1}M, P_d)
}.
\end{multline*}
Since $P_d$ is a finitely generated projective $\Lambda$-module, 
\[
{\rm Ext}^1_\Lambda(\Omega^{d-1}M, P_d) =
{\rm Ext}^{d}_\Lambda(M, P_d) =0
.
\]
Thus 
${\rm Hom}_\Lambda(f_d,P_d)$ is surjective.
So
there exists a homomorphism $g_d$
from $P_{d-1}$ into $P_d$ such that
the composition $g_d \circ f_d$ is the identity on $P_d$.
Therefore
the short exact sequence 
$0 \to P_d \to P_{d-1} \to \Omega^{d-1}M \to 0$
splits,
and
hence
$\Omega^{d-1}M$ is a direct summand of the projective module
$P_{d-1}$,
which is a contradiction.
\end{proof}

\begin{defn}
For a ring $R$ and a subclass $\fM$ of ${\rm Mod}R$, 
we define the assertion ${\sf P}_R(\fM)$
that means that,
for any $M\in\fM$,
if 
${\rm Ext}^{\geq 1}_R(M,R)=0$
then
$M$ is projective.
\end{defn}

\begin{remark}
\label{Noether ring}
For any Noetherian ring $\Lambda$
with finite global dimension
and
any finitely generated $\Lambda$-module $M$,
there is a projective precover of $M$
which is finitely generated.
So the above proof works for any Noetherian ring
of finite global dimension.
Therefore,
for  any Noetherian ring $\Lambda$
of finite global dimension, 
${\sf P}_\Lambda({\rm mod}\Lambda)$.
\end{remark}

It is known that
any path algebra $KQ$
over an algebraically closed field $K$, 
even when the quiver $Q$ is not finite,
is hereditary, that is,
its global dimension is not larger than $1$
(see e.g. \cite[\S8.2]{GR}).
So
any path algebra over an algebraically closed field 
is an algebra with finite global dimension.
Therefore 
{\bfsc Theorem \ref{fdA}} implies the following.

\begin{cor}
\label{fqa}
Suppose that 
$K$ is an algebraically closed field
and
$Q$ is a finite acyclic quiver.
Then
${\sf P}_{KQ}({\rm mod}KQ)$.
In particular, 
for any finitely generated $KQ$-module $M$,
if 
${\rm Ext}^{1}_{KQ}(M, KQ)=0$
then
$M$ is projective.
\end{cor}

%It is known that
%a path algebra of a cyclic quiver is Noetherian.
%So the following holds from 
%{\itsc Remark} {\rmsc \ref{Noether ring}].
%
%\begin{cor}
%Suppose that $Q$ is a cyclic quiver
%and $K$ is an algebraically closed field.
%Then
%${\sf P}_{KQ}({\rm mod}KQ)$ holds.
%In particular, 
%for any finitely generated $KQ$-module $M$,
%if 
%${\rm Ext}^{1}_{KQ}(M, KQ)=0$
%then
%$M$ is projective.
%\end{cor}

\begin{remark}
A finite quiver of the form
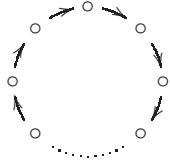
\[
%\text{\input{circle0.tex}}
%\ \ \ \ 
%%%\[
%%%\xymatrix@C=2pt@R=5pt{
%%%&&&& \circ \ar@/^4pt/[drrr]
%%%\\
%%%& \circ \ar@/^4pt/[urrr] &&& &&& \circ \ar@/^4pt/[ddr]
%%%\\
%%%\\
%%%\circ \ar@/^3pt/[uur] &&&& &&&& \circ \ar@/^3pt/[ddl]
%%%\\
%%%\\
%%%& \circ \ar@/^3pt/[uul] &&& &&& \circ \ar@{.}@/^10pt/[llllll]
%%%}
%%%\]
%%%\[
%%%\begin{xy}
%%%(13,24) *{\circ}="A", 
%%%(26,16.5) *{\circ}="B",
%%%(26,1.5) *{\circ}="C",
%%%(0,16.5) *{\circ}="D",
%%%(0,1.5) *{\circ}="E",
%%%(13,-6) ="F"
%%%\ar @/^5pt/ "A";"B" ,
%%%\ar @/^4pt/ "B";"C" ,
%%%\ar @/^4pt/ "E";"D" ,
%%%\ar @/^5pt/ "D";"A" ,
%%%\ar @/^6pt/ @{.} "C";"F" ,
%%%\ar @/^5.5pt/ @{.} "F";"E" 
%%%\end{xy}
%%%\]
\begin{xy}
(10,20) *++={\circ} ="A", 
(17,17) *++={\circ}="B",
(20,10) *++={\circ}="C",
(17,3) *++={\circ}="D",
(3,3) *++={\circ}="E",
(0,10) *++={\circ}="F",
(3,17) *++={\circ}="G",
(10,0) ="H",
\ar @/^2pt/ "A";"B" ,
\ar @/^2pt/ "B";"C" ,
\ar @/^2pt/ "C";"D" ,
\ar @/^2pt/ @{.} "D";"H" ,
\ar @/^2pt/ @{.} "H";"E" ,
\ar @/^2pt/  "E";"F" ,
\ar @/^2pt/ "F";"G" 
\ar @/^2pt/ "G";"A" 
\end{xy}
\]
is called a cyclic quiver.
Since the path algebra of a cyclic quiver 
$Q$ over an algebraically closed field $K$ 
is Noetherian
with finite global dimension,
it follows from {\itsc Remark} {\rmsc \ref{Noether ring}}
that 
${\sf P}_{KQ}({\rm mod}KQ)$.
\end{remark}

\medskip

We can extend the above corollary to 
{\em some} infinite quivers.
To introduce such infinite quivers explicitly,
we define the following notions.

\begin{defn}\begin{enumerate}
\item
A quiver $P=(P_0,P_1)$ is called a {\em subquiver} of 
a quiver $Q=(Q_0,Q_1)$
if
$P_0$ and $P_1$ are subsets of $Q_0$ and $Q_1$
respectively
(hence,
for any $a\in P_1$, $s(a)$ and $t(a)$ belong to $P_0$).

\item
For a quiver $Q$, 
a subquiver $P$ of $Q$,
a field $K$
and
a $K$-linear representation $\cX$ of $Q$,
the $K$-linear representation $\cX\restriction P$ of 
the quiver $P$ is called 
the {\em restricted representation of $\cX$ by $P$} 
if 
for every $v\in P_0$ and $a\in P_1$,
$(\cX\restriction P)_v=\cX_v$
and 
$(\cX\restriction P)_a=\cX_a$.

\item
For a quiver $Q=(Q_0,Q_1)$ and a subset $P_0'$ of $Q_0$,
the {\em closure of $P_0'$ under $Q$}
is the subquiver $\overline{P_0'}^Q=
\left(
\left( \overline{P_0'}^Q \right)_0 ,
\left( \overline{P_0'}^Q \right)_1
\right)$
of the quiver $Q$
such that
\begin{multline*}
\left( \overline{P_0'}^Q \right)_0 :=
\big\{
v\in Q_0:
\text{ there exists a path from a member of $P_0'$} 
\\
\text{to the vertex $v$ through the quiver $Q$}
\big\}
\end{multline*}
and
\[
\left( \overline{P_0'}^Q \right)_1 :=\set{a\in Q_1:
%\text{ there are $v$ and $w$ in $P_0$
%such that
%$a$ is an arrow from $v$ to $w$}
s(a)\in \left( \overline{P_0'}^Q \right)_0
}.
\]
%$\overline{P_0'}^Q$ denotes the closure of $P_0'$ under $Q$.

A subquiver $P$ of a quiver $Q$ is called
a {\em closed subquiver} of $Q$ if
$P$ is a closure of some subset of $Q_0$ under $Q$.
A subquiver $P$ of a quiver $Q$ is called
a {\em finite closed subquiver} of $Q$ if
$P$ is a closed subquiver of $Q$
and it is also a finite quiver.

%\item
%For a quiver $Q$ and $v\in Q_0$,
%$ \overline{\{v\}}^Q$ is the closure of $\{v\}$ under $Q$.

%\item
%For a quiver $Q$ and a closed subquiver $P$ of $Q$,
%$\cX(Q,P)$ denotes the $K$-linear representation of $Q$
%such that:
%for each $v\in P_0$, 
%$\cX(Q,P)_v:=K$; 
%for each $v\in Q_0\setminus P_0$,
%$\cX(Q,P)_v :=\{0_K\}$; 
%for each $a\in P_1$,
%$\cX(Q,P)_a$ is identity map on $K$;
%and 
%for each $a\in Q_1\setminus P_1$,
%$\cX(Q,P)_a$ is the trivial map
%from $\cX(Q,P)_{s(a)}$ into $\cX(Q,P)_{t(a)}$.

\end{enumerate}
\end{defn}

\begin{prop}
\label{closed quiver}
Suppose that
$K$ is a field,
$Q$ is a quiver,
$\cX$ is a $K$-linear representation of $Q$
such that
${\rm Ext}^1_{KQ}( G(\cX) , KQ)=0$,
and
$P$ is a closed subquiver of $Q$.
Then 
${\rm Ext}^1_{KP}( G(\cX\restriction P) , KP)=0$.
\end{prop}

\begin{proof}
Let $S$ be the functor from
the category 
${\rm Rep}_K Q$ into the category ${\rm Rep}_K P$
such that, 
for each $K$-linear representation $\cY$ of $Q$,
$S(\cY):=\cY\restriction P$,
and
let $T$ be the functor 
from ${\rm Rep}_K P$ into ${\rm Rep}_K Q$
such that, 
for each $K$-linear representation $\mathcal Z$ of $P$,
$T(\mathcal Z)$ is the $K$-linear representation of $Q$
such that
$T(\mathcal Z)_v := \mathcal Z_v$
for every $v\in P_0$,
$T(\mathcal Z)_v$ is the trivial $K$-vector space
for every $v\in Q_0\setminus P_0$,
$T(\mathcal Z)_a:= \mathcal Z_a$ 
for every $a\in P_1$,
and
$T(\mathcal Z)_a$ is the unique $K$-linear map
from the trivial $K$-vector space 
$T(\mathcal Z)_{s(a)}$
into 
the $K$-vector space 
$T(\mathcal Z)_{t(a)}$
for every $a\in Q_1\setminus P_1$.
%where $s(a)$ and $t(a)$ are the source 
%and the target of the arrow $a$ respectively.
We notice that
both $S$ and $T$ are exact functors,
and the functor $T$ is a right adjoint of the restricted functor $S$.
Moreover,
since $P$ is a closed subquiver of $Q$,
$T$ is well-defined,
that is,
the above $T(\mathcal Z)$
is certainly a $K$-representation of $Q$.
We also notice that
the composition $S\circ T$ is the identity functor
over ${\rm Rep}_K P$,
and 
\[
KP = \bigoplus_{v\in P_0} e_v KP
=
\bigoplus_{v\in P_0} e_v KQ
,
\]
which implies that
$T(F(KP))$ is a direct summand of $F(KQ)$.
Note that 
$G(T(F(KP)))$ is just $KP$ as a $KQ$-module, 
so it follows from our assumption that
${\rm Ext}^1_{KQ}( G(\cX) , KP )=0$.

${\rm Ext}^1_{KQ}( G(\cX) , KP )=0$ 
means that
any short exact sequence of $K$-linear representations of $Q$
of the form
\[
\xymatrix@C=15pt{
0 \ar[r] & T(F(KP)) \ar[r] & \cE \ar[r]^-{\varphi}
&  \cX \ar[r] & 0
}
\]
splits.
We note that
in such a short exact sequence, 
for any $v\in Q_0\setminus P_0$,
$\cE_v = \cX_v$ and 
$\varphi_v$ is an automorphism of $\cX_v$
(because $T(F(KP))_v$ is the trivial $K$-vector space).
So, 
for any short exact sequence 
of $K$-linear representations of $P$
of the form
\[
L': \ 
\xymatrix@C=15pt{
0 \ar[r] & F(KP) \ar[r] & \cE' \ar[r] & 
\cX\restriction P \ar[r] & 0
},
\]
there exists a short exact sequence 
of $K$-linear representations of $Q$
of the form
\[
L: \ 
\xymatrix@C=15pt{
0 \ar[r] & T(F(KP)) \ar[r] & \cE \ar[r] & 
\cX \ar[r] & 0
}
\]
such that
$S(L)=L'$.
Therefore,
it follows that
any short exact sequence of $K$-linear representations of $P$
of the form
\[
\xymatrix@C=15pt{
0 \ar[r] & F(KP) \ar[r] & \cE' \ar[r] & 
\cX\restriction P \ar[r] & 0
}
\]
splits,
which is equivalent to say that
${\rm Ext}^1_{KP}( G(\cX\restriction P) ,KP)=0$.
\end{proof}

\begin{prop}
\label{ext not vanish}
Suppose that 
$K$ is a field,
$Q$ is an acyclic quiver
that contains the quiver 
\[
\xymatrix@C=15pt{
v_0
& v_1 \ar[l] 
& v_2 \ar[l]
& \cdots \ar[l]
& v_n \ar[l]
& v_{n+1} \ar[l]
& \cdots \ar[l]
}
\]
as a subquiver, 
%such that,
%for each $n\in\omega$,
%$\overline{v_n}^Q$ is a finite acyclic quiver,
and
$\cX$ is a
finite dimensional $K$-linear representation of $Q$
such that, 
for each $n\in\omega$,
$\cX_{v_n}\neq\{0_K\}$,
and
$\cX\restriction \overline{\set{v_n}}^Q$ is a direct sum
%of $l_n$ many copies of 
%the corresponding $K$-linear representation $F(e_{v_n}KQ)$ of $e_{v_n}KQ$,
%where $l_n$ is the number of the arrows 
%from some vertex $v_{m}$, $m\in\omega$, to $v_n$ in $Q_1$,
of finitely many copies of 
the corresponding $K$-linear representation $F(e_{v_n}KQ)$ of $e_{v_n}KQ$.
%for each $v\in Q_0$,
%if $v$ does not belong to any quiver 
%$\overline{\set{v_n}}^Q$, $n\in \omega$,
%then $\cX_v=\{0_K\}$.
Then 
${\rm Ext}^1_{KQ}(G(\cX),KQ)\neq 0$.
\end{prop}

In \cite[Definition 3.5]{Enochsetc:flatflat}, 
some type of quivers is defined, 
which is called rooted.
It is true that 
a quiver $Q$ is rooted iff $Q$ does not contain 
the quiver $A_\infty^{\leftarrow}$ as a subquiver
\cite[Proposition 3.6]{Enochsetc:flatflat}.
So a quiver that satisfies the assumption of the proposition 
is not rooted.
For example, 
let $Q$ be the following quiver
\[
\xymatrix@C=15pt@R=15pt{
 w_0 \ar[d] & w_1 \ar[d] & w_2 \ar[d] & w_3 \ar[d] 
& w_4 \ar[d] 
\\
v_0
& v_1 \ar[l] 
& v_2 \ar[l]
& v_3 \ar[l]
& v_4 \ar[l]
& \cdots \ar[l]  
}
\]
and let $P:= \overline{\set{v_n: n\in\omega}}^Q$.
Then the quiver $Q$ is a non-rooted quiver, 
and the quiver $P$ is different from $Q$,
in fact, $P_0 =\set{v_n:n\in\omega}$.
For another example, 
let $Q$ be the following quiver 
%\[
%\xymatrix@C=20pt@R=5pt{
%& w_0 \ar@/_3pt/[dl] && w_2 \ar@/_3pt/[dl] &&
%w_5 \ar@/_3pt/[dl] &
%\\
%v_0
%& v_1 \ar[l]
%& v_2 \ar[l] \ar@/_3pt/[ul]
%& v_3 \ar[l] \ar@/^3pt/[dl]
%& v_4 \ar[l] \ar@/_3pt/[ul]
%& v_5 \ar[l]  \ar@/^3pt/[dl]
%& \cdots \text{ ,} \ar[l] \ar@/_3pt/[ul]
%\\
%&& w_1 \ar@/^3pt/[ul] && w_4 \ar@/^3pt/[ul] &
%}
%\]
\[
\xymatrix@C=15pt@R=15pt{
&  w_1  & w_2  & w_3 & w_4 
& \ 
\\
v_0
& v_1 \ar[l] \ar[u]
& v_2 \ar[l] \ar@<-0.7ex>[u] \ar@<0.7ex>[u]
& v_3 \ar[l] \ar[u]
& v_4 \ar[l] \ar@<-0.7ex>[u] \ar@<0.7ex>[u]
& \cdots \ar[l]  
}
\]
and let $P:= \overline{\set{v_n: n\in\omega}}^Q$.
Then the quiver $P$ is equal to the quiver $Q$ in this case,
and 
$P_0 \setminus \set{v_n:n\in\omega} = \set{w_n:n\in\omega} $.

\begin{proof}
Let $P:= \overline{\set{v_n: n\in\omega}}^Q$.
By {\bfsc Proposition \ref{closed quiver}},
it suffices to show that
${\rm Ext}^1_{KP}(G(\cX\restriction P),KP)\neq 0$.
Since 
\[
KP = 
\left( \bigoplus_{n\in \omega} e_{v_n} KP \right)
\oplus
\left( \bigoplus_{v\in P_0 \setminus \set{v_n:n\in\omega}} 
e_v KP \right)
\]
and 
\[
{\rm Ext}^1_{KQ}(M,N_0 \oplus N_1)=
{\rm Ext}^1_{KQ}(M,N_0) \oplus {\rm Ext}^1_{KQ}(M,N_1)
\]
in general,
it suffices to show that
${\rm Ext}^1_{KP}(G(\cX\restriction P), 
\displaystyle \bigoplus_{n\in\omega} e_{v_n} KP )\neq 0$. 

For each $n\in\omega$,
let 
\[
d_n := \max_{i\in\omega}
\left| \set{p:\text{ $p$ is a path from $v_i$ to $v_n$ on $P$}} \right|
,
\]
when such maximum exists as a finite number, or
$d_n:=\infty$ otherwise.
Notice that,
for each $n\in\omega$,
the dimension of the $K$-vector space $F(e_{v_n} KQ)_{v_0}$
is equal the number of paths from $v_n$ to $v_0$.
So, 
if infinitely many $d_n$ were larger than $1$,
then the dimension of $\cX_{v_0}$ had to be infinite.
Thus, 
for all but finitely many $n\in\omega$,
$d_n=1$.
Therefore,
without loss of generality 
we may assume that, 
for every $n\in\omega$, $d_n=1$.
Hence
there is a $d\in\omega\setminus\{0\}$
such that,
for any $n\in\omega$, 
\[
\cX\restriction \overline{\set{v_n}}^P
%$
%is isomorphic to 
%$
=
\displaystyle
\bigoplus_{d} F(e_{v_n}KP)
,
\]
where the last term is 
the outer direct sum
of $d$ many copies of
$F(e_{v_n}KP)$.
(Notice that
$d$ is the dimension of $\cX_{v_n}$.)
For each $n\in\omega$, 
let $a_n$ be the unique arrow from $v_{n+1}$ to $v_n$,
and,
for each $v\in P_{0}$,
let
\[
m(v):=\min\set{m\in\omega : 
\text{ there is a path from $v_m$ to $v$}}
.
\]
Then,
for any $v\in P_{0}$ and 
$n\geq m(v)$,
any path from $v_n$ to $v$ is of the form
$a_{n-1}\cdots a_{m(v)} p'$,
for some path $p'$ from $v_{m(v)}$ to $v$ in $P$.
Thus 
\[
\cX\restriction P
%$
%is isomorphic to 
%$
=
\displaystyle \bigoplus_{d} \cX^0
,
\]
where $\cX^0$ is the $K$-linear representation of $P$
such that:
for each $v\in P_0$,
$\cX^0_v$ is the $K$-vector space with basis
the set of all paths from $v_{m(v)}$ to $v$;
for each $n\in\omega$,
$\cX^0_{a_n}$ is the $K$-linear map from 
$\cX^0_{v_{n+1}}$ onto $\cX^0_{v_n}$
such that
$\cX^0_{a_n}(e_{v_{n+1}})=e_{v_n}$;
and,
for each $a\in P_1\setminus\set{a_n:n\in\omega}$,
$\cX^0_a$ is the $K$-linear map from 
$\cX^0_{s(a)}$ onto $\cX^0_{t(a)}$
such that,
for each path $p$ from $v_{m(s(a))}$ to $s(a)$,
$\cX^0_{a}(p):= pa$.
Since 
\[
{\rm Ext}^1_{KP}(\bigoplus_{i\in I}M_i,KP)=
\prod_{i\in I}{\rm Ext}^1_{KP}(M_i,KP)
\]
in general,
it suffices to show that
${\rm Ext}^1_{KP}(G(\cX^0), 
\displaystyle \bigoplus_{n\in\omega} e_{v_n} KP )\neq 0$.

To see this,
let $\pi$ be the canonical $KP$-epimorphism
from $\displaystyle \bigoplus_{n\in\omega} e_{v_n} KP$ onto $G(\cX^0)$ 
such that,
for each $n\in \omega$,
$\pi(e_{v_n}):= e_{v_n}$,
and,
for each $v\in P_0$ and each path $p$ in $P$ ending in $v$ of the form
$
p= a_{n-1} \cdots a_{m(v)} p'
$,
%where $p'$ is 
%either a stationary path,
%or
%a non-stationary path 
%whose first arrow is different from $a_{n-1}$,
$\pi(p):=p'$.
Then
${\rm Ker}(\pi)$ is the $KP$-submodule
of $\displaystyle \bigoplus_{n\in\omega} e_{v_n} KP$
which is generated by the
set
\[
%\begin{array}{r@{\,}l}
%A:= &
\set{e_{v_m}- a_n \cdots a_m :
m,n\in\omega, m \leq n}
%\\[0.5em]
%& \cup
%\set{ p :
%v\in Q_0\setminus P_0, 
%\text{ $p$ is a path of $Q$ from $v$}}
.
%\end{array}
\]
Applying ${\rm Hom}_{KP}(-,KP)$ to the
exact sequence
\[
\xymatrix@C=10pt{\displaystyle
0 \ar[r] &
{\rm Ker}(\pi) \ar[rrr]^-{{\rm id}_{{\rm Ker}(\pi)}}
&&& {\displaystyle \bigoplus_{n\in\omega} e_{v_n} KP} \ar[r]
& G(\cX^0) \ar[r]
& 0
},
\]
we obtain the exact sequence
\begin{multline*}
\xymatrix@C=10pt{
0 \ar[r] &
{\rm Hom}_{KP}( G(\cX^0) ,KP) \ar[r]
& \ }
\\
\xymatrix@C=10pt{{\rm Hom}_{KP}(
{\displaystyle \bigoplus_{n\in\omega} e_{v_n} KP} ,KP ) 
\ar[rrrrrr]^-{ {\rm Hom}_{KP}({\rm id}_{{\rm Ker}(\pi)},KP)}
&&&&&& {\rm Hom}_{KP}({\rm Ker}(\pi),KP)
}.
\end{multline*}
Then
\[
{\rm Ext}^1_{KP}( G(\cX^0) , 
{\displaystyle \bigoplus_{n\in\omega} e_{v_n} KP} )
=
{\rm Hom}_{KP}({\rm Ker}(\pi),KP ) 
\Big/
{\rm Im}( {\rm Hom}_{KP}({\rm id}_{{\rm Ker}(\pi)},KP) )
.
\]

For each non-stationary path $b_{n}\cdots b_0$ of $P$,
we fix 
the notation $b_{i}\cdots b_0$ by induction on 
$i\leq n$
in such a way that
\[
b_0 \cdots b_0 :=b_0
\]
and, for $i\leq n$, 
\[
b_{i+1}\cdots b_0 := b_{i+1} b_{i}\cdots b_0
.
\]
For each $m,n\in \omega$ with $m\leq n$,
%and 
%every path $p=a_{n}\cdots a_m$ of $Q$ from $v_{n+1}$ to $v_m$,
define
\[
\varphi(e_{v_m}- a_n \cdots a_m)
:= \sum_{i=0}^{n-m} a_{m+i} \cdots a_m
.
\]
%and,
%for each $v\in Q_0\setminus P_0$ 
%and any path $p$ of $Q$
%from $v$,
%define 
%\[
%\varphi(p)
%= 0_{KQ}
%.
%\]
Then,
for each $l,m,n\in\omega$
with $l<m\leq n$,
%and
%two paths 
%$p = a_{n}\cdots a_{m}$ from $v_{n+1}$ to $v_m$
%and 
%$q = a_{m-1}\cdots a_l$ from $v_m$ to $v_l$,
\[
\begin{array}{r@{\ = \ }l}
\varphi(e_{v_{m}}- a_{n} \cdots a_m) a_{m-1} \cdots a_l
& \displaystyle
\left( \sum_{i=0}^{n-m} a_{m+i} \cdots a_m \right)
a_{m-1} \cdots a_l
\\[2em]
& \displaystyle
\left( \sum_{i=0}^{n-l} a_{l+i} \cdots a_l \right)
-
\left( \sum_{i=0}^{m-1-l} a_{l+i} \cdots a_l \right)
\\[2em]
&
\varphi(e_{v_l}- a_{n} \cdots a_l)
-
\varphi(e_{v_l}- a_{m-1} \cdots a_l)
%\\[1em]
%=
%\varphi(\big(e_{v_l}- a_m \cdots a_l \big)
%-
%\big( e_{v_l}- a_{n-1} \cdots a_l \big) )
\end{array}
\]
and
\[
\begin{array}{r@{\ = \ }l}
(e_{v_{m}}- a_{n} \cdots a_m) a_{m-1} \cdots a_l
& a_{m-1} \cdots a_l - a_{n} \cdots a_l
\\[1em]
& 
\left( e_{v_l} - a_{n} \cdots a_l \right)
-
\left( e_{v_l} - a_{m-1} \cdots a_l \right)
.
\end{array}
\]
Thus, we can extend $\varphi$ to 
a $KP$-homomorphism from ${\rm Ker}(\pi)$ into $KP$.
To finish the proof, 
it is sufficient to show that 
$\varphi$ is not in 
${\rm Im}( {\rm Hom}_{KP}({\rm id}_{{\rm Ker}(\pi)},KP) )$.

Assume it is, and let 
$\psi \in {\rm Hom}_{KP}(
\displaystyle \bigoplus_{n\in\omega} e_{v_n} KP, KP)$
be
such that 
\[
\varphi
=
{\rm Hom}_{KP}({\rm id}_{{\rm Ker}(\pi)},KP) (\psi)
=
\psi\restriction {\rm Ker}(\pi)
.
\]
For each $n\in\omega$,
\[
\begin{array}{r@{\ = \ }l}
\psi(e_{v_0}) - \psi(e_{v_{n+2}})a_{n+1} \cdots a_0
& \psi(e_{v_0}) - \psi( a_{n+1} \cdots a_0)
\\[1em]
& \psi(e_{v_0} -  a_{n+1} \cdots a_0)
\\[1em]
& \displaystyle \sum_{i=0}^{n+1} a_i \cdots a_0.
\end{array}
\]
Therefore, 
for {\em every} $n\in \omega$,
$\psi(e_{v_0})$ belongs to the set
\[
\left(\sum_{i=0}^{n} a_i \cdots a_0\right)
+ 
KP \left(a_{n+1} \cdots a_0\right)
.
\]
However, this is a contradiction
because 
$\psi(e_{v_0})$ 
have to belong to $KP$.
\end{proof}

\begin{thm}
\label{fqa cor}
Suppose 
that $K$ is an algebraically closed field,
and
$Q$ is a connected quiver such that,
for any finite subset $P_0'$ of $Q_0$,
the closure of $P_0'$ under $Q$ is a finite acyclic quiver.
Then
${\sf P}_{KQ}({\rm rep}_K Q)$.
\end{thm}

For example,
the following quivers satisfy the assumption of the theorem:
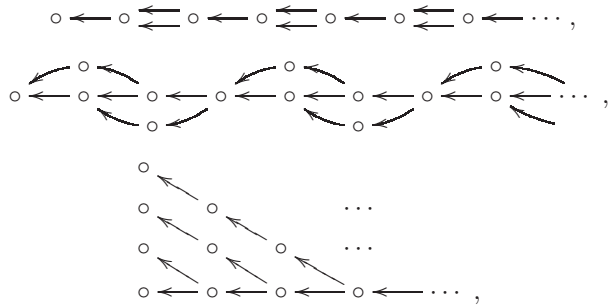
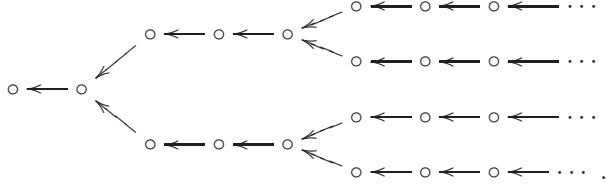
\[
\xymatrix@C=15pt{
\circ
& \circ \ar[l] 
& \circ \ar@<-0.7ex>[l] \ar@<0.7ex>[l]
& \circ \ar[l]
& \circ \ar@<-0.7ex>[l] \ar@<0.7ex>[l]
& \circ \ar[l]
& \circ \ar@<-0.7ex>[l] \ar@<0.7ex>[l]
& \cdots \ar[l]
},
\]
\[
\xymatrix@C=15pt@R=0pt{
& \circ \ar@/_3pt/[dl] &&& \circ \ar@/_3pt/[dl] &&&
\circ \ar@/_3pt/[dl] 
\\
\circ 
& \circ \ar[l]
& \circ \ar[l] \ar@/_3pt/[ul]
& \circ \ar[l] \ar@/^3pt/[dl]
& \circ \ar[l] 
& \circ \ar[l]  \ar@/_3pt/[ul]
& \circ \ar[l] \ar@/^3pt/[dl]
& \circ \ar[l]
& \cdots \text{ ,} \ar[l] \ar@/_3pt/[ul]
\\
&& \circ \ar@/^3pt/[ul] &&& \circ \ar@/^3pt/[ul] 
&&& \ \ \ \ \ \ar@/^3pt/[ul]
}
\]
\[
\xymatrix@C=15pt@R=5pt{
\circ %\ar@{.}[u]
\\
\circ & \circ \ar[lu] & & \cdots 
\\
\circ & \circ \ar[lu] & \circ \ar[lu] & \cdots 
\\
\circ & \circ \ar[l] \ar[lu] & \circ \ar[l] \ar[lu]  
& \circ \ar[l] \ar[lu] &\cdots  \text{ ,}\ar[l] 
}
\]
\[
\xymatrix@C=15pt@R=0pt{
&&&&& \circ \ar[dl] & \circ \ar[l] & \circ \ar[l] & \cdots \ar[l]
\\
&& \circ \ar[ddl] & \circ \ar[l] & \circ \ar[l] 
\\
&&&&& \circ \ar[ul] & \circ \ar[l] & \circ \ar[l] & \cdots \ar[l]
\\
\circ & \circ \ar[l] & 
\\
&&&&& \circ \ar[dl] & \circ \ar[l] & \circ \ar[l] & \cdots \ar[l]
\\
&& \circ \ar[uul] & \circ \ar[l] & \circ \ar[l] 
\\
&&&&& \circ \ar[ul] & \circ \ar[l] & \circ \ar[l] & \cdots 
\text{ .} \ar[l]
}
\]
Note that any infinite quiver as in
the assumption of the theorem
%{\bfsc Theorem \ref{fqa cor}}
contains at least one of the following quivers as a subquiver:
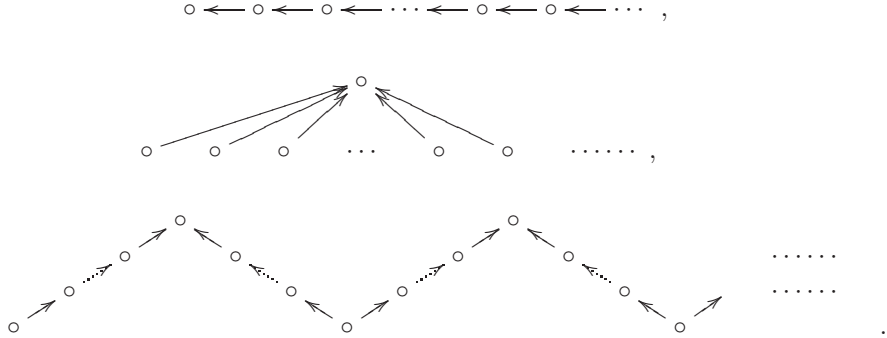
\[
%\text{\raisebox{-1em}{
\xymatrix@C=15pt{
\circ
& \circ \ar[l] 
& \circ \ar[l] 
& \cdots \ar[l]
& \circ \ar[l]
& \circ \ar[l]
& \cdots \ar[l]
}\text{ , }
%}}
\]
\[
\xymatrix@C=15pt@R=15pt{
&&& \circ & &&&
\\
\circ \ar[urrr]%^-{a_0} 
& \circ \ar[urr]%_-{a_1} %\ar[l]^-{b_1}
&  \circ \ar[ur]%_-{a_2} %\ar[l]^-{b_2}
& \cdots %\ar[l]^-{b_3}
%& \cdots 
& \circ \ar[ul]%_-{a_n} %\ar[l]^-{b_n}
& \circ \ar[ull]%_-{a_n} %\ar[l]^-{b_n}
& \cdots \cdots \text{ ,} %\ar[l]^-{b_{n+1}} 
}
\]
\[
\xymatrix@C=10pt@R=3pt{
&&&& \circ  && && && \circ 
\\
&&& \circ \ar[ur] && \circ \ar[lu] && && \circ \ar[ur]  && \circ \ar[ul]
&&&& \cdots \cdots
\\
&& \circ \ar@{.>}[ur] && && \circ \ar@{.>}[ul] && \circ \ar@{.>}[ur] &&& & \circ \ar@{.>}[ul] && &\cdots \cdots
\\
&\circ \ar[ur] &&& &&& \circ \ar[ur] \ar[ul] &&& &&& \circ \ar[ur] \ar[ul]
&&& \text{.}
}
\]

\begin{proof}
This theorem has been proved when $Q$ is a finite quiver
in {\bfsc Corollary \ref{fqa}}.
Suppose that
$Q$ is an infinite quiver,
and 
$\cX$ is a finite dimensional $K$-linear
representation of $Q$
such that
${\rm Ext}^1_{KQ}( G(\cX) ,KQ)=0$.
We show that $\cX$ is projective.
Since 
\[
{\rm Ext}^1_{KQ}(\bigoplus_{i\in I}M_i,KQ)=
\prod_{i\in I}{\rm Ext}^1_{KQ}(M_i,KQ)
\]
in general,
without loss of generality
we may assume that
$\cX$ is indecomposable.

By {\bfsc Proposition \ref{closed quiver}},
%for every finite subset $P_0'$ of $Q_0$,
%letting $P=(P_0,P_1)$ be the closure of $P_0'$ under $Q$,
for every finite closed subquiver $P$ of $Q$,
${\rm Ext}^1_{KP}( G(\cX\restriction P) ,KP)=0$.
Therefore,
by ${\sf P}_{KP}({\rm mod}KP)$,
$\cX\restriction P$ is projective.
It is known that
any indecomposable projective $KP$-module is of the form
$e_v KP$ for some $v\in P_0$
\cite[\S III.2]{ASS}.
Since $P$ is a closed subquiver of $Q$,
the underlying set of $e_v KP$ is equal to $e_v KQ$,
and 
$F(e_v KQ)\restriction P = F(e_v KP)$.
So, 
since $\cX$ is finite dimensional, 
$\cX\restriction P$ is isomorphic to a direct sum of 
finitely many $K$-linear representations
of the form $F(e_v KQ)$
for $v\in P_0$.
Therefore,
since $\cX$ is indecomposable, 
only one of the following statements hold:
\begin{enumerate}
\item
$\cX$ is isomorphic to 
$F(e_v KQ)$ for some $v\in Q_0$,
or

\item
$Q$ contains the quiver 
\[
\xymatrix@C=15pt{
v_0
& v_1 \ar[l] 
& v_2 \ar[l]
& \cdots \ar[l]
& v_n \ar[l]
& v_{n+1} \ar[l]
& \cdots \ar[l]
}
\]
as a subquiver
such that, 
for each $n\in\omega$,
$\cX_{v_n}\neq\{0_K\}$,
and
$\cX\restriction \overline{\set{v_n}}^Q$ is a direct sum
%of $l_n$ many copies of $F(e_{v_n}KQ)$,
%where $l_n$ is the number of the arrows from $v_{n+1}$ to $v_n$ in $Q_1$,
of finitely many copies of $F(e_{v_n}KQ)$.
%for each $v\in Q_0$,
%if $v$ does not belong to any quiver 
%$\overline{\set{v_n}}^Q$, $n\in\omega$,
%then $\cX_v=\{0_K\}$.
\end{enumerate}
By {\bfsc Proposition \ref{ext not vanish}}
and the assumption that ${\rm Ext}^1_{KQ}( G(\cX) ,KQ)=0$,
$\cX$ is isomorphic to 
$F(e_v KQ)$ for some $v\in Q_0$. 
By our assumption, %\marginpar{Add an explanation insted of `old Prop 1.9'.}
$ \overline{\{v\}}^Q$ is a finite acyclic quiver,
so
$e_v$ is an idempotent of $KQ$.
Hence
$e_v KQ$ is projective,
so is $\cX$.
\end{proof}

\subsection{Martin's Axiom}

Martin's Axiom was introduced by Martin and Solovay
\cite{MartinSolovay}.
This axiom cannot be neither proved nor refuted 
from axiomatic set theory
$\ZFC$, 
so it is consistent with $\ZFC$.
Martin's Axiom can be considered as 
a generalization of the Baire category theorem
(see e.g. \cite[Theorem III.4.7]{Kunen:new}).
$\MA_{\aleph_1}$ denotes
Martin's Axiom for $\aleph_1$ many dense sets.
In this paper
we use 
{\sf UP}\footnote{This notation follows \cite{EM}
but it is not  that common in set theory.
},
which is one combinatorial consequence from $\MA_{\aleph_1}$.

\begin{defn}
\begin{enumerate}
%\item
%In this paper,
%we adopt {\em ordinals} as the von Neumman ordinals,
%that is, 
%an ordinal $\alpha$ 
%means the set of ordinals less than $\alpha$.
%So for ordinals $\alpha$ and $\beta$,
%$\alpha$ is less than $\beta$ iff $\alpha\in\beta$.
%$\omega$ is the set of all finite ordinals
%(non-negative integers),
%$\omega_1$ is the least uncountable ordinal 
%(which is a cardinal).
%$\Lim$ denotes the class of all limit ordinals.

\item
A {\em ladder system (on $\omega_1$)}
is 
a sequence $\seq{C_\alpha:\alpha\in\omega_1\cap \Lim}$ 
such that
\begin{itemize}
\item
for each $\alpha\in\omega_1\cap \Lim$,
$C_\alpha$ is a cofinal subset of $\alpha$,
that is,
\[
\forall \xi \in\alpha\exists \eta\in C_\alpha
\big( \xi\in \eta\big)
,
\]
and

\item
$C_\alpha$ is of order type $\omega$,
that is,
the elements of $C_\alpha$ can be enumerated
as $\set{\zeta^\alpha_n:n\in\omega}$ increasingly,
that is,
for every $m,n\in\omega$, 
if $m\in n$, then 
$\zeta^\alpha_m \in \zeta^\alpha_n$.
\end{itemize}

\item
A {\em coloring} $\seq{d_\alpha:\alpha\in\omega_1\cap \Lim}$
of a ladder system 
$\seq{C_\alpha:\alpha\in\omega_1\cap \Lim}$
is a sequence of functions 
such that 
the domain of each $d_\alpha$ is $C_\alpha$.

\item
We say that
a function $f$ with domain $\omega_1$ 
{\em uniformizes} 
a coloring 
$\langle d_\alpha:\alpha\in\omega_1\cap \Lim \rangle$
of a ladder system 
$\seq{C_\alpha:\alpha\in\omega_1\cap \Lim}$,
$C_\alpha=\set{\zeta^\alpha_n:n\in\omega}$,
if
for every $\alpha\in\omega_1\cap \Lim$,
the restricted function 
$f\restriction C_\alpha$ 
of $f$ by $C_\alpha$ 
is equal to 
the function $d_\alpha$
for all but finitely many points,
that is,
there exists an $N\in\omega$
such that, 
for any $n\in\omega\setminus N$,
$f(\zeta^\alpha_n)=d_\alpha(\zeta^\alpha_n)$.

\item
The assertion {\sf UP} 
means that,
for any 
%ladder system 
%$\seq{C_\alpha:\alpha\in\omega_1\cap \Lim}$, 
%sequence $\seq{X_\beta : \beta\in\omega_1}$
%of countable sets,
%and 
%coloring $\seq{d_\alpha:\alpha\in\omega_1\cap \Lim}$
%of $\seq{C_\alpha:\alpha\in\omega_1\cap \Lim}$,
sequence $\seq{X_\beta : \beta\in\omega_1}$
of countable sets
and 
any coloring $\seq{d_\alpha:\alpha\in\omega_1\cap \Lim}$
of a ladder system
$\seq{C_\alpha:\alpha\in\omega_1\cap \Lim}$,
whenever
$d_\alpha(\zeta^\alpha_n)$ belongs to $X_{\zeta^\alpha_n}$
for any $\alpha\in\omega_1\cap \Lim$
and $n\in\omega$,
there exists a function with domain $\omega_1$
which uniformizes the coloring 
$\seq{d_\alpha:\alpha\in\omega_1\cap \Lim}$.

\end{enumerate}
\end{defn}

%%The following is all of what we need about Matin's Axiom.
%For the proof of the following result, 
%see e.g.
%\cite[Proposition VI.4.6]{EM}
%or
%\cite[\S3]{Yorioka:MArec}.

\begin{thm}[Devlin-Shelah \text{\cite[Theorem 5.2]{DevlinShelah: weak diamond}}]
$\MA_{\aleph_1}$
implies {\sf UP}.
\end{thm}

The assertion {\sf UP} was inspired by Shelah's proof
that 
$\MA_{\aleph_1}$ 
implies the existence of
a non-free Whitehead group 
\cite[Theorem 3.5]{Shelah:W}
(see also \cite{E Shelah}).
%To finish this subsection,
%we present the following result which will be used in the proof of
%our main results.
%
%\begin{prop}
%\label{laddr claim}
%For any ladder system 
%$\seq{C_\delta:\delta\in\omega_1\cap \Lim}$,
%$C_\delta=\set{\zeta^\delta_n:n\in\omega}$,
%and $\alpha\in\omega_1\cap \Lim$,
%there exists 
%a sequence $\seq{N_\delta:\delta\in\alpha\cap\Lim}$
%of members of $\omega$
%such that
%the set
%\[
%\Big\{C_\alpha\Big\}
%\cup
%\Big\{
%\set{\zeta^\delta_n:n\geq N_\delta}
%:\delta\in\alpha\cap \Lim
%\Big\}
%\]
%is pairwise disjoint.
%
%\end{prop}
%
%
%\begin{proof}
%We enumerate the set $\alpha\cap \Lim$
%by $\set{\delta_i:i\in \omega}$
%(which is not an increasing enumeration in general).
%By induction on $i\in\omega$,
%we can choose an $N_i\in\omega$ such that
%the set
%$\set{\zeta^{\delta_i}_n:n\geq N_{i}}$
%is disjoint from the set
%\[
%C_\alpha
%\cup
%\bigcup_{j<i}
%\set{\zeta^{\delta_j}_n:n\geq N_{j}}
%.
%\]
%This is because 
%each set
%\[
%C_{\delta_i}
%\cap
%\left(
%C_\alpha
%\cup
%\bigcup_{j<i}
%\set{\zeta^{\delta_j}_n:n\geq N_{j}}
%\right)
%\]
%is finite.
%\end{proof}

%%%%%%%%%%%%%%%%%
\subsection{Trlifaj's construction}
\label{ods}

%Every proof in \S\ref{inf gen pa} is 
%fairly self-contained, 
%so the reader does not need to be familiar with
%Trlifaj's construction.
%To fix our notation and understand 
%our construction better,
%Trlifaj's construction is presented below.

%We follow the notation of outer direct sums
%in \cite[I.2.]{EM}.
%For a family $\set{M_i:i\in I}$ of modules,
%the {\rm product module}
%$\displaystyle \prod_{i\in I} M_i$ is the module
%whose underling set is the set of
%functions $f$ with domain $I$ such that
%for each $i\in I$, 
%$f(i)$ belongs to the set $M_i$,
%and 
%the operations are defined coordinate-wise.
%For a member $f$ of the product $\displaystyle \prod_{i\in I} M_i$,
%the {\em support} ${\rm supp}(f)$ of $f$ is defined by the set
%\[
%\set{i\in I: f(i)\neq 0_{M_i}}
%.
%\]
%The outer direct sum
%$\displaystyle \bigoplus_{i\in I} M_i$
%of a family $\set{M_i:i\in I}$ of modules
%is the submodule of the product
%module $\displaystyle \prod_{i\in I} M_i$ 
%which
%consists of the members of the set
%$\displaystyle \prod_{i\in I} M_i$ 
%whose supports are finite.

In this paper,
our modules are built by modifying
Trlifaj's construction.
%for a sequence 
%%$\seq{F^\xi:\xi\in\omega_1}$
%of modules of length $\omega_1$.
As every proof in \S\ref{inf gen pa} is fairly self-contained,
the reader does not need to be familiar with this construction.
Trlifaj's construction is a quotient module of the outer direct sum 
$\displaystyle \bigoplus_{\xi\in\omega_1} F^\xi$
of some sequence
$\seq{F^\xi:\xi\in\omega_1}$
of modules, defined in 
\cite[Definition 1.1]{Tr ES}
and 
\cite[Notation 5.3]{HT},
which seems to be inspired by 
Shelah's solution of Whitehead Problem
\cite{Shelah:W}.
%(see also \cite{E Shelah}).
To fix our notation and understand our construction better,
Trlifaj's construction is presented as follows.

Let $R$ be a ring
with identity
and 
let 
\[
\xymatrix@C=30pt{
F_0 \ar[r]^-{\text{\normalsize $f_0$}} 
& F_1 \ar[r]^-{\text{\normalsize $f_1$}} 
& \cdots \ar[r]^-{\text{\normalsize $f_{n-1}$}}
& F_n \ar[r]^-{\text{\normalsize $f_n$}} 
& F_{n+1} \ar[r]^-{\text{\normalsize $f_{n+1}$}} 
& \cdots 
}
\]
be a countable direct system of $R$-modules.
Let 
$\seq{C_\alpha:\alpha\in\omega_1\cap\Lim}$
be a ladder system 
such that
\[
C_\alpha=\set{\zeta^\alpha_n:n\in\omega}
\]
is an increasing enumeration
and 
assume that,
for each $n\in\omega$,
$\zeta^\alpha_n$ is of the form $\delta+n+1$
for some $\delta\in\alpha\cap (\{0\}\cup \Lim)$
(there is such a ladder system).
Define $F^0 := \{0_R\}$; 
for each $\gamma\in\omega_1\setminus\Lim$
with $\gamma=\delta+n_\gamma+1$
for some $\delta\in\gamma\cap (\{0\}\cup \Lim)$ and $n_\gamma\in\omega$,
define 
$F^\gamma:=F_{n_\gamma}$; 
and,
for each $\delta\in\omega_1\cap \Lim$,
define
$F^\delta:=
\displaystyle
\bigoplus_{n\in\omega}F_n$.
So, for each member $x$ of 
the outer direct sum 
$\displaystyle \bigoplus_{\xi\in\omega_1} F^\xi$,
$x$ forms a finite support function with domain 
included in $\omega_1$
and,
for each
$\alpha\in\omega_1$,
$x(\alpha)$ belongs to $F^\alpha$.
Hence, if $\delta\in\omega_1\cap \Lim$,
then
$x(\delta)$ belongs to
the outer direct sum
$F^\delta:=
\displaystyle
\bigoplus_{n\in\omega}F_n$,
which also forms a finite support function with domain 
included in $\omega$.
For
each $\delta\in\omega_1\cap \Lim$,
define
the $R$-submodule
\begin{multline*}
G_\delta:=
\left\langle
\bigg\{x\in \bigoplus_{\xi\in\omega_1} F^\xi :
\text{ for some $n\in\omega$, }
\supp(x)=\set{\zeta^\delta_n, \delta},
\right.
\\
\supp(x(\delta))=\set{n, n+1},
x(\zeta^\delta_n)=x(\delta)(n),
\\
\left.
\text{ and }
x(\delta)(n+1)= f_n(x(\delta)(n))
\bigg\}
\right\rangle_{R}
\end{multline*}
of the $R$-module
$\displaystyle \bigoplus_{\xi\in\omega_1} F^\xi$,
and
define
\[
I_{\omega_1}:=\sum_{\delta\in\omega_1\cap\Lim}G_\delta
,
\]
which is an $R$-submodule of the $R$-module 
$\displaystyle \bigoplus_{\xi\in\omega_1} F^\xi$.
Trlifaj's construction
is the quotient $R$-module
$\displaystyle 
\left.\bigoplus_{\xi\in\omega_1} F^\xi
\right/ I_{\omega_1}$
of the $R$-module
$\displaystyle \bigoplus_{\xi\in\omega_1} F^\xi$
by the $R$-submodule $I_{\omega_1}$.
%For each $x\in \displaystyle \bigoplus_{\xi\in\omega_1} F^\xi$,
%$x+ I_{\omega_1}$ denotes the equivalent class of $x$ 
%in this quotient.
%\marginpar{Remove one sentence here.}
Trlifaj applied this construction for a non-left perfect ring 
\cite{Tr ES},
and 
Herbera-Trlifaj applied it to analyze some classes
of modules called Kaplansky classes
or deconstructible classes
\cite{HT}.
For further properties of this module,
see
\cite[\S5]{HT}.

%\S\ref{inf simple} provides 
%an infinitely generated non-projective module
%over a path algebra of the quiver
%of $A_\infty$ type,
%and
%\S\ref{fin} provide 
%an infinitely generated non-projective module
%over a path algebra of the Kronecker quiver.
%Two of them follow the above notation.

%We will notice that
%each $F^\xi$ is constructed by an outer direct sum
%of some family of modules too.
%So for each member $f$ of 
%the module $\displaystyle \bigoplus_{\xi\in\omega} F^\xi$
%and index $\xi\in\omega_1$,
%$f(\xi)$ is also a member of some outer direct sum of modules,
%that is,
%forms a function.

%%%%%%%%%%%%%%%%%%%%%
\section{Some infinitely generated modules of path algebras}
\label{inf gen pa}

%\subsection{Definition of the infinite quiver and its module}

%\footnote{\tt file name: \jobname}

Throughout this section,
we fix a ladder system 
$\seq{C_\alpha:\alpha\in\omega_1\cap\Lim}$
such that
\[
C_\alpha=\set{\zeta^\alpha_n:n\in\omega}
\]
is an increasing enumeration
and,
for each $n\in\omega$,
$\zeta^\alpha_n$ is of the form $\delta+n+1$
for some $\delta\in\alpha\cap( \{0\}\cup \Lim )$.
We note that, 
for any $\alpha,\beta\in\omega_1\cap \Lim$
and $m,n\in\omega$,
if $\zeta^\alpha_m=\zeta^\beta_n$ then 
$m=n$.
For $\gamma\in\omega_1\setminus(\{0\}\cup \Lim)$,
let $n_\gamma\in\omega$ be the unique integer such that
$\gamma=\delta+n_\gamma+1$ 
for some (unique) $\delta\in \omega_1\cap (\{0\}\cup \Lim)$.

For each subsection of this section,
we deal with some quiver $Q$ and 
build a non-projective $KQ$-module $M_{\omega_1}$.
For each quiver $Q$ in each subsection, 
we use the following notation.
For each $v\in Q_0$,
$e_{v}$ denotes the path of length $0$
from the vertex $v$ 
(to itself).
For $\gamma\in\omega_1 \setminus \Lim$,
$\alpha\in\omega_1\cap \Lim$ and $n\in\omega$,
let 
$F^\gamma=F^{\alpha,n}:=KQ$,
%$F^\gamma$ and $F^{\alpha,n}$
%be the $KQ$-module
%$\displaystyle
%\underset{\omega}{\bigoplus} KQ$,
%which is the outer direct sum of $\omega$ many copies of $KQ$,
and 
let 
$F^\alpha$ be the outer direct sum 
$\displaystyle \bigoplus_{n\in\omega} F^{\alpha,n}$.
%of the $KQ$-modules $F^{\alpha,n}$.
For $\gamma\in\omega_1\setminus\Lim$
%$j\in\omega$ 
and 
$v\in Q_0$,
let 
$e^{\gamma}_{v}$
%$e^{\gamma,j}_{v}$
be the member
of the outer direct sum 
$\displaystyle \bigoplus_{\xi\in\omega_1} F^\xi$
of $KQ$-modules 
such that
\[
\supp(e^{\gamma}_{v})=\{\gamma\},
\ \ \ 
e^{\gamma}_{v}(\gamma)=e_{v}
.
\]
%\[
%\supp(e^{\gamma,j}_{v})=\{\gamma\},
%\ \ \ 
%\supp(e^{\gamma,j}_{v}(\gamma))=\{j\},
%\ \ \ 
%e^{\gamma,j}_{v}(\gamma)(j)=e_{v}
%.
%\]
For $\alpha\in\omega_1\cap \Lim$,
$n\in\omega$ 
%$n,j\in\omega$ 
and $v\in Q_0$,
let
$e^{\alpha,n}_{v}
%$e^{\alpha,n,j}_{v}
\in 
\displaystyle \bigoplus_{\xi\in\omega_1} F^\xi$
be
such that
\[
\supp(e^{\alpha,n}_{v})=\{\alpha\},
\ \ \ 
\supp(e^{\alpha,n}_{v}(\alpha))=\{n\},
\ \ \ 
e^{\alpha,n}_{v}(\alpha)(n)=
e_{v}
.
\]
%\[
%\supp(e^{\alpha,n,j}_{v})=\{\alpha\},
%\ \ \ 
%\supp(e^{\alpha,n,j}_{v}(\alpha))=\{n\},
%\]
%\[
%\supp(e^{\alpha,n,j}_{v}(\alpha)(n))=\{j\},
%\ \ \ 
%e^{\alpha,n,j}_{v}(\alpha)(n)(j)=
%e_{v}
%.
%\]

\begin{remark}
\label{direct sum F xi}
The set 
\[
\set{e^{\gamma}_{v} ,
e^{\alpha,n}_{v} :
\gamma\in\omega_1\setminus\Lim,
%j\in\omega,
v\in Q_0,
\alpha\in \omega_1 \cap\Lim,
n\in\omega
}
\]
%\[
%\set{e^{\gamma,j}_{v} ,
%e^{\alpha,n,j}_{v} :
%\gamma\in\omega_1\setminus\Lim,
%j\in\omega,
%v\in Q_0,
%\alpha\in \omega_1 \cap\Lim,
%n\in\omega
%}
%\]
is linearly independent 
with respect to $KQ$
in $\displaystyle \bigoplus_{\xi\in\omega_1} F^\xi$.
\end{remark}

%%%%%%%%%%%%%%%%%%%%%%%%%
\subsection{On a quiver of $A_\infty$ type}
\label{inf simple}
Throughout this subsection,
let 
$K$ be a field, 
and
$Q$ the quiver $A_\infty^{\leftarrow}$ as follows:
\[
\xymatrix@C=30pt{
0 
& 1 \ar[l]_-{\text{\normalsize $a_0$}} 
& \cdots \ar[l]_-{\text{\normalsize $a_1$}} 
& n \ar[l]_-{\text{\normalsize $a_{n-1}$}} 
& n+1 \ar[l]_-{\text{\normalsize $a_n$}} 
& \cdots \ar[l]_-{\text{\normalsize $a_{n+1}$}} 
} \ ,
\]
that is,
the set $Q_0$ of vertices
is the set of all non-negative integers
and 
the set $Q_1$ of arrows
is defined by
\[
\set{
\xymatrix@C=18pt{n & n+1
\ar[l]_-{\text{\normalsize $a_{n}$}}  : n\in\omega}
}
.
\]
Since $Q_0$ is infinite, 
$KQ$ does not have an identity.
By simplifying the notation
in this subsection,
%for each $\gamma\in\omega\setminus\Lim$,
%we denote
%\[
%e^\gamma_{n_\gamma} :=
%e^{\gamma,0}_{n_\gamma}
%,
%\]
%and,
for each $\alpha\in\omega_1\cap\Lim$ and $n\in\omega$,
\[
e^{\alpha}_{n}
:=
e^{\alpha,0}_{n}
.
\]
For each $\alpha\in\omega_1\cap \Lim$,
define
\[
G_\alpha  := 
\left\langle
\left\{
e^{\zeta^\alpha_n}_{n}
-
e^{\alpha}_{n}
+
e^{\alpha}_{n+1}
a_{n}
:
n\in\omega
\right\}\right\rangle_{KQ}
,
\]
\[
I_{\omega_1}:=\sum_{\xi\in\omega_1\cap\Lim}G_\xi .
\]
For each $x\in \displaystyle \bigoplus_{\xi\in\omega_1} F^\xi$,
$x+ I_{\omega_1}$ denotes the equivalence class of $x$ 
in the quotient module 
$\displaystyle
\left.
\displaystyle \bigoplus_{\xi\in\omega_1} F^\xi
\right/I_{\omega_1}$.
For each $\xi\in\omega_1+1$,
define
the $KQ$-module
$M_\xi$ 
by
\[
\left\langle
\set{e^\gamma_{n_\gamma} + I_{\omega_1} :
\gamma\in \xi\setminus\Lim}
\cup
\set{
e^{\alpha}_{n} + I_{\omega_1} :
\alpha\in \xi\cap\Lim,
n\in\omega}
\right\rangle_{KQ}
,
\]
which is considered as a $KQ$-submodule of 
the quotient module
$\displaystyle
\left.
\displaystyle \bigoplus_{\xi\in\omega_1} F^\xi
\right/I_{\omega_1}
$.

%\begin{remark}
%\label{inf ctbl ds}
%For each $\xi\in\omega_1$, 
%$M_\xi$ is countably generated,
%and
%$\seq{M_\xi:\xi\in\omega_1}$
%forms a direct system of $KQ$-modules
%with direct limit $M_{\omega_1}$.
%\end{remark}

\begin{remark}
\label{inf ctbl I basis}
The set 
$\left\{
e^{\zeta^\alpha_n}_{n}
-
e^{\alpha}_{n}
+
e^{\alpha}_{n+1}
a_{n} :
\alpha\in\omega_1\cap \Lim,
n\in\omega
\right\}$
is linearly independent 
with respect to $KQ$
in $\displaystyle \bigoplus_{\xi\in\omega_1} F^\xi$.
%and it generates a direct summand of 
%$\displaystyle \bigoplus_{\xi\in\omega_1} F^\xi$.
\end{remark}

In this paper,
${\displaystyle \bigoplus_{\omega_1} KQ}$
denotes
the outer direct sum of $\omega_1$ many copies of $KQ$,
which is considered as a $KQ$-module.

\begin{claim}
\label{inf ctbl M alpha 1}
${\rm Ext}^1_{KQ}( 
M_{\omega_1}, 
{\displaystyle \bigoplus_{\omega_1} KQ} )
\neq 0$.
In particular, 
$M_{\omega_1}$ is not a projective $KQ$-module.
\end{claim}

\begin{proof}
$F_{\omega_1}$ denotes 
the $KQ$-module
\[
\left\langle
\set{e^\gamma_{n_\gamma} :
\gamma\in \omega_1\setminus\Lim}
\cup
\set{
e^{\alpha}_{n}  :
\alpha\in \omega_1\cap\Lim,
n\in\omega}
\right\rangle_{KQ}
.
\]
Applying ${\rm Hom}_{KQ}( - , 
{\displaystyle \bigoplus_{\omega_1} KQ} )$
to the exact sequence
\[
\xymatrix@C=10pt{\displaystyle
0 \ar[r] &
I_{\omega_1} \ar[rr]^-{{\rm id}_{I_{\omega_1}}}
&& F_{\omega_1} \ar[r]
& M_{\omega_1} \ar[r]
& 0
},
\]
we obtain the exact sequence 
\begin{multline*}
\xymatrix@C=10pt{
0 \ar[r] &
{\rm Hom}_{KQ}(M_{\omega_1},
{\displaystyle \bigoplus_{\omega_1} KQ} ) \ar[r]
& \ }
\\
\xymatrix@C=10pt{{\rm Hom}_{KQ}(F_{\omega_1},
{\displaystyle \bigoplus_{\omega_1} KQ} ) 
\ar[rrrrrr]^-{ {\rm Hom}_{KQ}({\rm id}_{I_{\omega_1}},
\underset{\omega_1}{\bigoplus}KQ)}
&&&&&& {\rm Hom}_{KQ}(I_{\omega_1},
{\displaystyle \bigoplus_{\omega_1} KQ}
)
}.
\end{multline*}
Then
\[
{\rm Ext}^1_{KQ}(M_{\omega_1},
{\displaystyle \bigoplus_{\omega_1} KQ} )
=
\left. {\rm Hom}_{KQ}(I_{\omega_1},
{\displaystyle \bigoplus_{\omega_1} KQ} ) \right/
{\rm Im}( {\rm Hom}_{KQ}({\rm id}_{I_{\omega_1}},
{\displaystyle \bigoplus_{\omega_1} KQ} ) )
.
\]
By {\itsc Remark} {\rmsc \ref{inf ctbl I basis}},
we can find
a $KQ$-homomorphism $\varphi$ 
in ${\rm Hom}_{KQ}( I_{\omega_1} ,
{\displaystyle \bigoplus_{\omega_1} KQ} )$
such that
for each $\alpha\in\omega_1\cap\Lim$ and $n\in\omega$,
\[
\varphi(
e^{\zeta^\alpha_n}_{n}
-
e^{\alpha}_{n}
+
e^{\alpha}_{n+1}
a_{n}) :=
e^\alpha_{n}
.
\]
We show that
$\varphi$ does not belong to 
${\rm Im}({\rm Hom}_{KQ}({\rm id}_{I_{\omega_1}},
{\displaystyle \bigoplus_{\omega_1} KQ} ) )$.

Assume that $\varphi\in {\rm Im}({\rm Hom}_{KQ}({\rm id}_{I_{\omega_1}},
{\displaystyle \bigoplus_{\omega_1} KQ} ) )$, %.
%Then there exists
%a $\psi$ in the set ${\rm Hom}_{KQ}(M_{\omega_1}, 
%{\displaystyle \bigoplus_{\omega_1} KQ} )$
and 
let 
$\psi$ be a homomorphism in 
${\rm Hom}_{KQ}( F_{\omega_1}, 
{\displaystyle \bigoplus_{\omega_1} KQ} )$
such that
\[
{\rm Hom}_{KQ}({\rm id}_{I_{\omega_1}},
{\displaystyle \bigoplus_{\omega_1} KQ} ) 
(\psi)=\psi\circ {\rm id}_{I_{\omega_1}}
= \psi\restriction I_{\omega_1}=\varphi
.
\]
We note that
for each $\gamma\in\omega_1$ and $n\in\omega$,
$\supp(\psi(e^{\gamma}_{n}))$
is a finite subset of $\omega_1$.
So
we can take an $\alpha\in\omega_1\cap \Lim$
such that, 
for every $\gamma\in\alpha$ and $n\in\omega$,
$\supp(\psi(e^{\gamma}_{n}))$
is a finite subset of $\alpha$\footnote{
This can be done by e.g. \cite[Exercise III.6.20]{Kunen:new}}.
For each $n\in\omega$,
\[
\psi(e^{\zeta^\alpha_n}_{n})
-
\psi(e^{\alpha}_{n})
+
\psi(e^{\alpha}_{n+1})
a_{n}
=
\psi(
e^{\zeta^\alpha_n}_{n}
-
e^{\alpha}_{n}
+
e^{\alpha}_{n+1}
a_{n})
= e^\alpha_n
.
\]
Therefore, by induction on $n\in\omega$,
\[
\begin{array}{r@{\ }c@{\ }l}
\psi(e^\alpha_0)
& = & 
\psi(e^{\zeta^\alpha_0}_{0})
-
e^{\alpha}_{0}
+
\psi(e^{\alpha}_{1})
a_{0}
\\[1em]
\rule{25pt}{0pt}
& = &
\psi(e^{\zeta^\alpha_0}_{0})
-
e^{\alpha}_{0}
+
\left(
\psi(e^{\zeta^\alpha_1}_1) - e^\alpha_1 + \psi(e^\alpha_2) a_1 
\right)
a_{0}
\\[1em]
\rule{25pt}{0pt}
& = &
\psi(e^{\zeta^\alpha_0}_{0})
-
e^{\alpha}_{0}
+
\psi(e^{\zeta^\alpha_1}_1) a_{0} 
- e^\alpha_1 a_{0} 
+ \psi(e^\alpha_2) a_1 a_{0}
\\[1em]
\rule{25pt}{0pt}
& = & \cdots
\\
\rule{25pt}{0pt}
& = & \displaystyle
\psi(e^{\zeta^\alpha_0}_0) 
- e^\alpha_0
+ \sum_{i=1}^{n} \psi(e^{\zeta^\alpha_i}_i)
a_{i-1}\cdots a_0
-
\sum_{i=1}^{n}e^\alpha_i
a_{i-1}\cdots a_0
\\
\rule{25pt}{0pt}
&  & 
+ \psi(e^{\alpha}_{n+1})
a_{n} \cdots a_0.
\end{array}
\]
Hence, 
for {\em every} $n\in\omega$,
since each 
$\supp(\psi(e^{\zeta^\alpha_i}_i))$ does not contain $\alpha$
as a member,
\[
\psi(e^\alpha_0 )(\alpha)
\not\in
K Q_{\leq n}
,
\]
where $KQ_{\leq n}$ is 
the $K$-subspace of $KQ$
generated by all paths of length $\leq n$.
This is a contradiction. 
%because 
%$\psi(e^\alpha_0 )(\alpha)$ 
%have to belong to $KQ$.
\end{proof}

The following is similar to 
\cite[Theorem XII 2.2,  Proposition XIII 0.2]{EM}.

\begin{thm}
\label{inf simple thm}
Suppose that
$K$ is a countable field.
Then
{\sf UP}
implies that
${\rm Ext}^1_{KQ}(M_{\omega_1},KQ)=0$.
In particular, 
${\sf P}_{KQ}({\rm Mod}KQ)$ fails.
\end{thm}

\begin{proof}
Applying ${\rm Hom}_{KQ}( - , KQ)$
to the exact sequence 
\[
\xymatrix@C=10pt{\displaystyle
0 \ar[r] &
I_{\omega_1} \ar[rr]^-{{\rm id}_{I_{\omega_1}}}
&& F_{\omega_1} \ar[r]
& M_{\omega_1} \ar[r]
& 0
},
\]
we obtain the exact sequence 
\begin{multline*}
\xymatrix@C=10pt{
0 \ar[r] &
{\rm Hom}_{KQ}(M_{\omega_1},KQ) \ar[r]
& \ }
\\
\xymatrix@C=10pt{
&
{\rm Hom}_{KQ}(F_{\omega_1}, KQ ) 
\ar[rrrrr]^-{ {\rm Hom}_{KQ}({\rm id}_{I_{\omega_1}}, KQ)}
&&&&& {\rm Hom}_{KQ}(I_{\omega_1}, KQ)
}.
\end{multline*}
Then
\[
{\rm Ext}^1_{KQ}(M_{\omega_1}, KQ )
=
\left. {\rm Hom}_{KQ}(I_{\omega_1}, KQ ) \right/
{\rm Im}( {\rm Hom}_{KQ}({\rm id}_{I_{\omega_1}}, KQ ) )
.
\]

Let $\varphi\in {\rm Hom}_{KQ}(I_{\omega_1},KQ) $.
We show that
$\varphi$ belongs to 
${\rm Im}( {\rm Hom}_{KQ}({\rm id}_{I_{\omega_1}}, KQ ) )$.
%We notice that,
%for  $\alpha\in\omega_1\cap\Lim$ and $n\in\omega$,
%\[
%\varphi(
%e^{\zeta^\alpha_n}_{n}
%-
%e^{\alpha}_{n}
%+
%e^{\alpha}_{n+1}
%a_{n}
%)
%\, 
%e_{n}
%=
%\varphi(
%e^{\zeta^\alpha_n}_{n}
%-
%e^{\alpha}_{n}
%+
%e^{\alpha}_{n+1}
%a_{n}
%)
%,
%\]
%and, 
%for any $m\in Q_0\setminus \{n\}$,
%\[
%\varphi(
%e^{\zeta^\alpha_n}_{n}
%-
%e^{\alpha}_{n}
%+
%e^{\alpha}_{n+1}
%a_{n}
%)
%\, 
%e_{m}
%=
%\varphi(0_{\underset{\xi\in\omega_1}{\bigoplus} F^\xi})
%=
%0_{KQ}
%.
%\]
For each $\alpha\in\omega_1\cap\Lim$
and 
$n\in\omega$,
define
\[
d_\alpha(\zeta^\alpha_n)
:=
\varphi(
e^{\zeta^\alpha_n}_{n}
-
e^{\alpha}_{n}
+
e^{\alpha}_{n+1}
a_{n}
)
.
\]
We notice that,
for each $n\in\omega$,
\[
\varphi(
e^{\zeta^\alpha_n}_{n}
-
e^{\alpha}_{n}
+
e^{\alpha}_{n+1}
a_{n}
)
\, 
e_{n}
=
\varphi(
e^{\zeta^\alpha_n}_{n}
-
e^{\alpha}_{n}
+
e^{\alpha}_{n+1}
a_{n}
)
,
\]
and, 
for any $m\in Q_0\setminus \{n\}$,
\[
\varphi(
e^{\zeta^\alpha_n}_{n}
-
e^{\alpha}_{n}
+
e^{\alpha}_{n+1}
a_{n}
)
\, 
e_{m}
=
\varphi(0_{\underset{\xi\in\omega_1}{\bigoplus} F^\xi})
=
%0_{\underset{\omega_1}{\bigoplus} KQ}
0_{KQ}.
\]
Thus,
for each $n\in\omega$,
$d_\alpha(\zeta^\alpha_n)$ belongs to the countable set
\[
\sum_{p \text{ : path in $Q$ ending in
$n$}}
K p
.
\]
Therefore,
by {\sf UP},
we can find a uniformization 
$f$ 
of the ladder system coloring
$\seq{d_\alpha:\alpha\in\omega_1\cap \Lim}$,
that is,
for each $\alpha\in\omega_1\cap \Lim$,
there is an $N_\alpha\in\omega$
such that,
for every $n\geq N_\alpha$,
$f(\zeta^\alpha_n)=d_\alpha(\zeta^\alpha_n)$.

For each $\alpha\in\omega_1\cap\Lim$
and $n\in\omega$,
define
\begin{itemize}
\item
$\psi(e^{\zeta^\alpha_n}_{n})
:= f(\zeta^\alpha_n)$,

\item
$\psi(e^{\alpha}_{n})
:= 0_{KQ}$
when $n\geq N_\alpha$,
and

\item
by downward induction on $n<N_\alpha$,
define
\[
\psi(e^{\alpha}_{n})
:=
\psi(e^{\zeta^\alpha_n}_{n})
+
\psi(e^{\alpha}_{n+1})
a_{n}
\\
-
\varphi(
e^{\zeta^\alpha_n}_{n}
-
e^{\alpha}_{n}
+
e^{\alpha}_{n+1}
a_{n}
)
.
\]
\end{itemize}
By {\em Remark} \ref{direct sum F xi},
$\psi$ can be extended to a $KQ$-homomorphism
from 
$F_{\omega_1}$
into $KQ$.
Therefore
\[
\psi\restriction I_{\omega_1}
=
\psi\circ {\rm id}_{I_{\omega_1}}
=
{\rm Hom}_{KQ}({\rm id}_{I_{\omega_1}}, KQ )(\psi)
=
\varphi
,
\]
which finishes the proof.
\end{proof}

\begin{remark}
By a similar argument to the one in the previous theorem,
it can be proved that
{\em 
if $K$ is a countable field
and
{\sf UP} holds,
then
$
{\rm Ext}^1_{KQ}(M_{\omega_1},
\displaystyle \bigoplus_{\omega} KQ)$ $=0$.
}
\end{remark}

\begin{remark}
\label{inf simple non v ext}
By a similar argument as in \cite[Lemma 4.3]{E Shelah},
we can show that 
{\em if $K$ is a countable field,
then
$\diamondsuit$\footnote{$\diamondsuit$ 
is one set theoretic axiom
consistent with $\ZFC$,
see e.g. \cite{Kunen:new}.} 
implies 
${\rm Ext}^1_{KQ}(M_{\omega_1},KQ)\neq0$.}
The main ingredient to prove this is the following fact.
\begin{claim}
Suppose that
${\rm Ext}_{KQ}(M_{\alpha+1}/M_\alpha,KQ)\neq 0$,
and let
\[
\xymatrix@C=15pt{\displaystyle
0 \ar[r] &
KQ \ar[r]
& C_\alpha \ar[r]^-{\pi}
& M_{\alpha} \ar[r]
& 0
}
\]
be a short exact sequence that
splits,
that is,
there exists
a homomorphism $\rho$ from $M_\alpha$ into $C_\alpha$
such that
$\pi\circ \rho= {\rm id}_{M_\alpha}$.
Then there exists a short exact sequence 
\[
\xymatrix@C=15pt{\displaystyle
0 \ar[r] &
KQ \ar[r]
& C_{\alpha+1} \ar[r]^-{\pi'}
& M_{\alpha+1} \ar[r]
& 0
}
\]
such that
\[
\pi'\restriction C_\alpha=\pi
\]
and
there is no homomorphism $\rho'$
from $M_{\alpha+1}$ into $C_{\alpha+1}$
such that
\[
\pi'\circ\rho'={\rm id}_{M_\alpha+1}
\ \ \ \text{ and } \ \ \ 
\rho'\restriction M_\alpha=\rho.
\]
\end{claim}

\noindent
Since $KQ$ is countable and {\bfsc Claim \ref{inf ctbl M alpha 1}} holds,
a similar argument as in \cite[Theorem 6.3]{E Shelah}
works well to show that
${\rm Ext}_{KQ}(M_{\omega_1},KQ)\neq 0$.
Moreover, 
by a similar argument as in \cite{HS},
we can show that,
{\em if $K$ is a countable field
and
there is 
a set $\set{S_\alpha:\alpha\in\omega_1}$ of 
pairwise disjoint stationary subsets of $\omega_1$
such that
$\diamondsuit_{S_\alpha}$ holds for each $\alpha\in\omega_1$,
then 
the cardinality of ${\rm Ext}_{KQ}(M_{\omega_1},KQ)$
is greater than $\aleph_1$.}
%\hfill
%$\dashv_{\underline{\text{{\itsc Remark} {\rmsc \ref{inf simple non v ext}}}}}$
\end{remark}

%%%%%%%%%%%%%%%%%%%%%%%%%%
\subsection{On a circular quiver}

In this subsection,
let 
$K$ be a field
and 
$Q$ the following quiver.
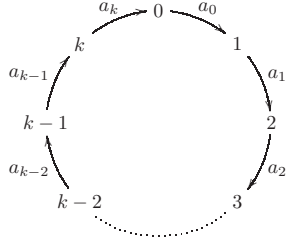
\[
%%%\text{\input{circle(arrow).tex}}
%%%\ \ \ \ \ \ \ \ 
\scalebox{0.75}{
\begin{xy}
(20,40) *++={0} ="A", 
(34,34) *++={1}="B",
(40,20) *++={2}="C",
(34,6) *++={3}="D",
(6,6) *+-={k-2}="E",
(0,20) *+-={k-1}="F",
(6,34) *++={k}="G",
(20,0) ="H",
\ar^{\text{\normalsize $a_0$}} @/^4pt/ "A";"B" ,
\ar^{\text{\normalsize $a_1$}} @/^4pt/ "B";"C" ,
\ar^{\text{\normalsize $a_2$}} @/^4pt/ "C";"D" ,
\ar @/^5pt/ @{.} "D";"H" ,
\ar @/^5pt/ @{.} "H";"E" ,
\ar^{\text{\normalsize $a_{k-2}$}} @/^4pt/  "E";"F" ,
\ar^{\text{\normalsize $a_{k-1}$}} @/^4pt/ "F";"G" 
\ar^{\text{\normalsize $a_k$}} @/^4pt/ "G";"A" 
\end{xy}
}
\]
Then the path $a_0 a_1 \cdots a_k$
is a path in $Q$ whose source and target are both
the vertex $0$.
We denote the path
\[
\left(a_0 a_1 \cdots a_k\right)^0
= e_0
,
\]
and, 
for each $n\in\omega$,
define the path
\[
\left(a_0 a_1 \cdots a_k\right)^{n+1}
= 
\left(a_0 a_1 \cdots a_k\right)^n
a_0 a_1 \cdots a_k
.
\]
Recall that 
$\displaystyle
\sum_{v\in Q_0}e_v$
is the identity of $KQ$.
%As in the previous subsection,
%for each 
%%$\gamma\in\omega_1\setminus\Lim$, 
%$\alpha\in\omega_1\cap\Lim$ and $n\in\omega_1$,
%we denote
%\[
%%e^{\gamma}_0
%%:= e^{\gamma,0}_0,
%%\ \ \ 
%e^{\alpha,n}_0 := e^{\alpha,n,0}_0.
%\]
For each $\alpha\in\omega_1\cap\Lim$, 
define
\[
G_\alpha  := 
\left\langle
\left\{
e^{\zeta^\alpha_n}_{0}
-
e^{\alpha,n}_{0}
+
e^{\alpha,n+1}_{0}
a_0 a_1 \cdots a_k :
n\in\omega
\right\}\right\rangle_{KQ}
,
\]
\[
I_{\omega_1}:=\sum_{\xi\in\omega_1\cap\Lim}G_\xi
,
\]
and,
for each $\xi\in\omega_1+1$,
define
the $KQ$-module
$M_\xi$ 
by
\[
\left\langle
\set{e^{\gamma}_{0} + I_{\omega_1} :
\gamma\in \xi\setminus\Lim}
\cup
\set{
e^{\alpha,n}_{0} + I_{\omega_1} :
\alpha\in \xi\cap\Lim,
n\in\omega}
\right\rangle_{KQ}
,
\]
which is considered as a $KQ$-submodule of 
the quotient module
$\displaystyle
\left.
\displaystyle \bigoplus_{\xi\in\omega_1} F^\xi
\right/I_{\omega_1}.
$

\begin{claim}
\label{2 M alpha 1}
${\rm Ext}^1_{KQ}( 
M_{\omega_1},\displaystyle \bigoplus_{\omega_1} KQ)
\neq 0$.
Therefore, $M_{\omega_1}$ is not projective.
\end{claim}

%%Though $KQ$ of \S\ref{inf simple} does not have an identity,
%%$KQ$ of this subsection has an identity,
%%hence
%It follows from this claim that
%$M_{\omega_1}$ is not projective.

\begin{proof}
This can be proved in a similar way as in
{\bfsc Claim \ref{inf ctbl M alpha 1}}.
To see this,
it suffices to replace 
the formula
\[
\varphi(
e^{\zeta^\alpha_n}_{n}
-
e^{\alpha}_{n}
+
e^{\alpha}_{n+1}
a_{n}) := e^\alpha_n
\]
by the formula
\[
\varphi(
e^{\zeta^\alpha_n}_{0}
-
e^{\alpha,n}_{0}
+
e^{\alpha,n+1}_{0}
a_0 a_1 \cdots a_k)
:= e^{\alpha,n}_{0}
\]
in the proof of {\bfsc Claim \ref{inf ctbl M alpha 1}}.
\end{proof}

Moreover, by a similar proof as 
{\bfsc Theorem \ref{inf simple thm}},
the following theorem can be proved.

\begin{thm}
\label{2 thm}
Suppose that
$K$ is a countable field.
Then
{\sf UP}
implies that
${\rm Ext}^1_{KQ}(M_{\omega_1},KQ)=0$.
In particular,
${\sf P}_{KQ}({\rm Mod}KQ)$ fails.
\end{thm}

\begin{remark}
\label{2 non v ext}
As in {\itsc Remark} {\rmsc \ref{inf simple non v ext}},
%if $K$ is countable, 
%then $\diamondsuit$ implies
%${\rm Ext}^1_{KQ}(M_{\omega_1},KQ)$
%$\neq 0$.
if $K$ is a countable field and
$\diamondsuit$ holds, 
then 
${\rm Ext}^1_{KQ}(M_{\omega_1},KQ)\neq 0$.
\end{remark}

%%%%%%%%%%%%%%%%%%
\subsection{Generalizations}

\begin{thm}
\label{cor final}
Suppose that
$K$ is a countable field
and 
$Q'$ is a quiver
that contains 
a subquiver $Q$ of one of the following types
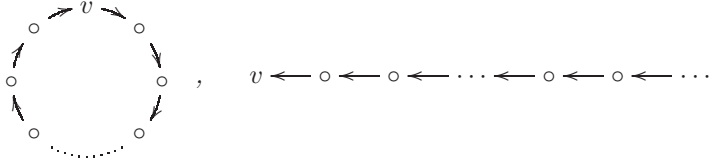
\[
\begin{xy}
(10,20) *++={v} ="A", 
(17,17) *++={\circ}="B",
(20,10) *++={\circ}="C",
(17,3) *++={\circ}="D",
(3,3) *++={\circ}="E",
(0,10) *++={\circ}="F",
(3,17) *++={\circ}="G",
(10,0) ="H",
\ar @/^2pt/ "A";"B" ,
\ar @/^2pt/ "B";"C" ,
\ar @/^2pt/ "C";"D" ,
\ar @/^2pt/ @{.} "D";"H" ,
\ar @/^2pt/ @{.} "H";"E" ,
\ar @/^2pt/  "E";"F" ,
\ar @/^2pt/ "F";"G" 
\ar @/^2pt/ "G";"A" 
\end{xy}
\text{\raisebox{3em}{ \ , \ \ \ $\xymatrix@C=15pt{
v
& \circ \ar[l]
& \circ \ar[l]
& \cdots \ar[l]
& \circ \ar[l]
& \circ \ar[l]
& \cdots \ar[l]
}$}}
\ \ \ 
\]
\noindent
in such a way that
the set of all paths in $Q'$ ending in $v$ is countable.
Then
{\sf UP}
implies the failure of
${\sf P}_{KQ'}({\rm Mod}KQ')$.
\end{thm}

%As seen in {\itsc Remark} {\rmsc \ref{ex A inf}},
%every countably infinite locally finite quiver
%that satisfies the assumption of
%{\bfsc Theorem \ref{fqa cor}}
%also satisfies the assumption of this theorem.

\begin{proof}
Let $M_{\omega_1}$ be one of the $KQ$-modules constructed before.
Then, $M_{\omega_1}$ 
can be considered 
as a $KQ'$-module
and, by a similar argument as before,
it can be proved that
${\rm Ext}^1_{KQ'}(M_{\omega_1},
\displaystyle \bigoplus_{\omega_1}KQ')\neq 0$,
and 
that
{\sf UP}
implies 
${\rm Ext}^1_{KQ'}(M_{\omega_1},KQ') =0$.
\end{proof}

\noindent
{\bfsc Acknowledgement}.
\
The authors thank Hiroyuki Minamoto, Izuru Mori 
and Kenta Ueyama
for useful comments about this research.
Especially, they provided us advice and information about 
{\bfsc Theorem \ref{fdA}},
{\itsc Remark} {\rmsc \ref{Noether ring}}
and
{\bfsc Proposition \ref{closed quiver}}.

The authors also appreciate the useful comments from the referee.

%%%%%%%%%%%%%%%%%%%%%%%%%%
%%%%%%%%%%%%%%%%%%%%%%%%%%


\begin{thebibliography}{00}

%\markboth{BIBLIOGRAPHY}{BIBLIOGRAPHY}

%\bibitem{AbrahamShelah: MA not BA}
%U. Avraham and S. Shelah.
%{\em Martin's axiom does not imply that 
%every two $\aleph_1$-dense sets of reals are isomorphic}. 
%Israel J. Math. 38 (1981), no. 1-2, 161--176

%\bibitem{AbrahamRubinShelah}
%U. Abraham, M. Rubin and S. Shelah.
%{\em On the consistency of some partition theorems 
%for continuous colorings, and the structure of 
%$\aleph_1$-dense real order types}.
%Ann. Pure Appl. Logic 29 (1985), no. 2, 123--206. 

%\bibitem{AT: partition properties}
%U. Abraham and S. Todor\v cevi\'c.
%{\em Partition properties of $\omega_1$ 
%compatible with CH}.
%Fund. Math. 152 (1997), 165--180.

\bibitem{AF}
F. W. Anderson and K.R. Fuller.
Rings and categories of modules. Second edition. 
Graduate Texts in Mathematics, 13. {\it Springer-Verlag, New York}, 1992.

%\bibitem{AsperoMota: c large}
%David Asper\'o and Miguel Angel Mota, 
%{\em Forcing consequences of PFA together with the continuum large}. 
%Trans. Amer. Math. Soc.
%367 (2015), no. 9, 6103--6129.

%\bibitem{AsperoMota: measuring c large}
%David Asper\'o and Miguel Angel Mota, 
%{\em Measuring club-sequences with a large continuum}.
%Preprint, 
%arXiv:1203.1238.

%\bibitem{AsperoMota: genMA}
%David Asper\'o and Miguel Angel Mota, 
%{\em A Generalization of Martin's Axiom}.
%Israel J. Math. 210 (2015), no. 1, 193--231. 
%
%\bibitem{AsperoMota: AM clubguess}
%David Asper\'o and Miguel Angel Mota, 
%{\em Separating club-guessing principles 
%in the presence of fat forcing axioms}. 
%Ann. Pure Appl. Logic 167 (2016), no. 3, 284--308

\bibitem{ARS}
M. Auslander, I Reiten and S. Smal\o.
Representation theory of Artin algebras.Cambridge Studies in Advanced Mathematics, 36.
{\it Cambridge University Press, Cambridge}, 
1995. 

\bibitem{ASS}
I. Assem, D. Simson and A. Skowro\' nski.
Elements of the representation theory of associative algebras. Vol. 1. 
Techniques of representation theory. 
London Mathematical Society Student Texts, 65. 
{\it Cambridge University Press, Cambridge}, 2006. 


%\bibitem{AR}
%M. Auslander and I. Reiten.
%On a generalized version of the Nakayama conjecture. 
%{\em Proc. Amer. Math. Soc}. 52 (1975), 69--74. 

%\bibitem{BahlekehFallah} 
%A. Bahlekeh and A. M. Fallah.
%Progress on the Auslander-Reiten conjecture. 
%{\em Bull. Aust. Math. Soc}. 93 (2016), no. 3, 433--440. 

%\bibitem{BFS}
%A. Bahlekeh, A. M. Fallah and S. Salarian. 
%On the Auslander-Reiten conjecture for algebras. 
%{\em J. Algebra} 427 (2015), 252--263.

%\bibitem{BKS}
%A. Bahlekeh, T. Kakaei and S. Salarian.
%On the Auslander-Reiten conjecture for Cohen-Macaulay rings 
%and path algebras. 
%{\it Comm. Algebra} 45 (2017), no. 1, 121--129.

%\bibitem{Bartoszynski: addN implies addM}
%T. Bartoszy\'nski. 
%{\em Additivity of measure implies 
%additivity of category}. 
%Trans. Amer. Math. Soc. 281 (1984), no. 1, 209--213.

%\bibitem{Bartoszynski: handbook}
%T. Bartoszy\'nski. 
%{\em Invariants of Measure and Category}. 
%Handbook of set theory. Vols. 1, 2, 3, 491--555, Springer, Dordrecht, 2010.

%\bibitem{BartoszynskiJudah: Book}
%T. Bartoszy\'nski and H. Judah. 
%{\em Set theory. 
%On the structure of the real line}. 
%A K Peters, Ltd., Wellesley, MA, 1995.

%\bibitem{Baumgartner: iterated forcing}
%J. Baumgartner.
%{\em Iterated forcing}. 
%Surveys in set theory, 1--59, London Math. Soc. Lecture Note Ser., 
%87, Cambridge Univ. Press, Cambridge, 1983.


%\bibitem{Baumgartner: PFA}
%J. Baumgartner.
%{\em Applications of the Proper Forcing Axiom,}
%Handbook of Set-Theoretic Topology, chapter 21, pp. 913-959.
%
%\bibitem{Bell: p}
%M. Bell.
%{\em On the combinatorial principle $P(\fc)$}. 
%Fund. Math. 114 (1981), no. 2, 149--157.
%
%\bibitem{Blass: handbook}
%A. Blass.
%{\em Combinatorial cardinal characteristics of the continuum}. 
%Handbook of set theory. Vols. 1, 2, 3, 395--489, Springer, Dordrecht, 2010.

%\bibitem{Bekkali note}
%M. Bekkali.
%{\em Topics in set theory. Lebesgue measurability, 
%large cardinals, forcing axioms, rho-functions. 
%Notes on lectures by Stevo Todorcevic}. 
%Lecture Notes in Mathematics, 1476. 
%Springer-Verlag, Berlin, 1991.


\bibitem{Benson:Book}
D. J. Benson.
Representations and cohomology. I. 
Basic representation theory of finite groups and associative algebras.
Cambridge Studies in Advanced Mathematics, 30. 
{\em Cambridge University Press, Cambridge}, 1991.


\bibitem{Brune:leftright}
H. Brune.
Some left pure semisimple ringoids which are not right pure semisimple. 
{\em Comm. Algebra} 7 (1979), no. 17, 1795--1803. 


%\bibitem{CelikbasTakahashi}
%O. Celikbas and R. Takahashi.
%Auslander-Reiten conjecture and Auslander-Reiten duality. 
%{\em J. Algebra} 382 (2013), 100--114. 

\bibitem{DevlinShelah: weak diamond}
K. Devlin and S. Shelah.
{\em A weak version of $\diamondsuit$ which follows from 
$2^{\aleph_0}<2^{\aleph_1}$}.
Israel J. Math. 29 (1978), no. 2-3, 239--247.

%\bibitem{DiPriscoTodorcevic: LRU}
%C. A. Di Prisco and S. Todor\v cevi\'c.
%{\em Perfect-set properties in $L(\mathbb R)[U]$}. 
%Adv. Math. 139 (1998), no. 2, 240--259. 

%\bibitem{EisworthMoore: iteration CH}
%T. Eisworth, J. T. Moore and D. Milovich
%{\em Iterated forcing and the Continuum Hypothesis}, 
%in Appalachian set theory 2006-2012, J. Cummings and E. Schimmerling, eds. 
%London Math Society Lecture Notes series, Cambridge University Press (2013). 


\bibitem{Enochsetc:projrep}
E. E. Enochs and S. Estrada.
Projective representations of quivers. (English summary) 
{\em Comm. Algebra} 33 (2005), no. 10,3467--3478. 

\bibitem{Enochsetc:flatcovers}
E. E. Enochs, S. Estrada, J. R. Garc\'ia Rozas and L. Oyonarte.
Flat covers of representations of the quiver $A_\infty$. 
{\em Int. J. Math. Math. Sci.} 2003, no. 70,4409--4419. 

%\bibitem{Enochsetc:hom1999}
%E. E. Enochs and I. Herzog.
%A homotopy of quiver morphisms with applications to representations. 
%{\em Canad. J. Math}. 51 (1999), no. 2, 294--308.

\bibitem{Enochsetc:flatflat}
E. E.  Enochs, L.  Oyonarte and  B. Torrecillas.
Flat covers and flat representations of quivers.
{\em Comm. Algebra} 32 (2004), no. 4,1319--1338. 


\bibitem{E Shelah}
P. C. Eklof.
Whitehead's problem is undecidable. 
{\it Amer. Math. Monthly} 83 (1976), no. 10, 775--788. 

\bibitem{EM}
P. C. Eklof and A. H. Mekler. 
Almost free modules. Set-theoretic methods. Revised edition. 
North-Holland Mathematical Library, 65. 
{\it North-Holland Publishing Co., Amsterdam}, 2002.


%\bibitem{Farah: embedding}
%I. Farah.
%{\em Embedding partially ordered sets into $\omega^\omega$}. 
%Fund. Math. 151 (1996), no. 1, 53--95.


%\bibitem{Farah: 1995 1}
%I. Farah.
%{\em Small forcing preserves Reflection}, 
%April 2,
%1995.


%\bibitem{Farah: 1995 2}
%I. Farah.
%{\em Preserving Reflection}, 
%April 18,
%1995.

%\bibitem{Farah: 1995}
%I. Farah.
%{\em ${\rm OCA}_{\aleph_1}+{\rm MA}_{\aleph_1}$ is consistent with 
%continuum being large}, 
%April 27,
%1995.


%\bibitem{Farah: 1995 3}
%I. Farah.
%{\em All gaps are Hausdorff}, 
%November 13,
%1995.



%\bibitem{Farah: OCA}
%I. Farah.
%{\em OCA and towers in $\mathcal P(N)/{\rm fin}$}. 
%Comment. Math. Univ. Carolin. 37 (1996), no. 4, 861--866. 

%\bibitem{Farah: quotients}
%I. Farah.
%{\em Analytic quotients: 
%theory of liftings for quotients over analytic ideals 
%on the integers}. 
%Mem. Amer. Math. Soc. 148 (2000), no. 702, xvi+177 pp. 

%\bibitem{Farah: inner}
%I. Farah.
%{\em All automorphisms of the Calkin algebra are inner}. 
%Ann. of Math. (2) 173 (2011), no. 2, 619--661.

%\bibitem{FischerTallTodorcevic}
%A. J. Fischer, F. D. Tall and S. Todor\v cevi\'c.
%{\em Forcing with a coherent Souslin tree and 
%locally countable subspaces of countably tight compact spaces},
%preprint.

\bibitem{Fremlin}
D. H. Fremlin.
Consequences of Martin's axiom. 
Cambridge Tracts in Mathematics, 84. 
{\em Cambridge University Press, Cambridge}, 1984. 

 \bibitem{G}
   P. Gabriel, 
   Auslander-Reiten sequences and representation-finite algebras, 
   {\it Representation theory, I 
   (Proc. Workshop, Carleton Univ., Ottawa, Ont., 1979)}, 
   1--71, Lecture Notes in Math., \textbf{831}, Springer, Berlin, 1980. 

\bibitem{GR}
P. Gabriel and A. V. Roiter. 
Representations of finite-dimensional algebras. 
With a chapter by B. Keller. Encyclopaedia Math. Sci., 73, 
{\it Algebra, VIII}, 1--177, {\it Springer, Berlin}, 1992. 

%\bibitem{GruenhageMashburn}
%G. Gruenhage and J. Mashburn.
%{\em On the decomposition of order-separable posets of 
%countable width into chains}.
%Order 16 (1999), no. 2, 171--177 (2000). 

\bibitem{HT}
D. Herbera and J. Trlifaj.
Almost free modules and Mittag-Leffler conditions. 
{\it Adv. Math}. 229 (2012), no. 6, 3436--3467. 

\bibitem{HS}
H. L. Hiller and S. Shelah.
Singular cohomology in $L$. 
{\it Israel J. Math}. 26 (1977), no. 3--4, 313--319. 

%\bibitem{Hohino:rcz}
%M. Hoshino.
%On algebras with radical cube zero. 
%{\em Arch. Math. (Basel)} 52 (1989), no. 3, 226--232.

%\bibitem{Jech: set theory}
%T. Jech.
%{\em Set theory. The third millennium edition, 
%revised and expanded}. 
%Springer Monographs in Mathematics, Springer-Verlag, 
%Berlin, 2003.

%\bibitem{Just: weak AT}
%W. Just.
%{\em A weak version of {\rm AT} from {\rm OCA}}. 
%Set theory of the continuum (Berkeley, CA, 1989), 281--291, 
%Math. Sci. Res. Inst. Publ., 26, Springer, New York, 1992. 

%\bibitem{Koszmider: side conditions}
%P. Koszmider.
%{\em Models as side conditions}. 
%Set theory (Cura?ao, 1995; Barcelona, 1996), 99--107, 
%Kluwer Acad. Publ., Dordrecht, 1998.

%\bibitem{Kunen: Set Theory}
%K. Kunen. 
%{\em Set theory. 
%An introduction to independence proofs}. 
%Studies in Logic and the Foundations of Mathematics, 
%102. North-Holland Publishing Co., 
%Amsterdam-New York, 1980. 
%
%\bibitem{Kunen: Set Theory new}
%K. Kunen. 
%{\em Set theory}. 
%Studies in Logic (London), 34. College Publications, London, 2011.

%\bibitem{Larson: Pmax OCA}
%Paul Larson, 
%{\em Showing OCA in Pmax-style extensions}. 
%Kobe J. Math. 18 (2001), no. 2, 115--126.


%\bibitem{LarsonTall: lcpn p}
%P. Larson and F. Tall.
%{\em Locally compact perfectly normal spaces may all be 
%paracompact}, 
%preprint. 

%\bibitem{LarsonTodorcevic: chain conditions}
%P. Larson and S. Todor\v cevi\'c.
%{\em Chain conditions in maximal models}. 
% Fund. Math. 168 (2001),  no. 1, 77--104. 
%
%\bibitem{LarsonTodorcevic: Katetov}
%P. Larson and S. Todor\v cevi\'c.
%{\em Kat\v etov's problem}. 
%Trans. Amer. Math. Soc. 354 (2002), no. 5, 1783--1791. 

\bibitem{Kunen:new}
K. Kunen.
Set theory. Studies in Logic (London), 34. 
{\em College Publications, London}, 2011.

%\bibitem{KunenTall: MA SH}
%K. Kunen and F. Tall.
%{\it Between Martin's axiom and Souslin's hypothesis},
%Fund. Math.
%102 (1979), no. 3, 173--181.

%\bibitem{Mahrt}
%N. Mahrt.
%Representations of the generalized Kronecker quiver 
%with countably many arrows. 
%{\em Proc. Amer. Math. Soc.} 137 (2009), no. 3, 815--824. 

%\bibitem{Mantese}
%F. Mantese. 
%Complements to projective almost complete tilting modules. 
%{\em Comm. Algebra} 33 (2005), no. 9, 2921--2940.

\bibitem{MartinSolovay}
D. Martin and R. Solovay.
Internal Cohen extensions. 
{\em Ann. Math. Logic} 2 1970 no. 2, 143--178. 

\bibitem{Minamoto}
H. Minamoto.
Ampleness of two-sided tilting complexes.
{\it Int. Math. Res. Not. IMRN} 2012, no. 1, 67--101.


%\bibitem{Moore: two OCA}
%J. T. Moore.
%{\em Open colorings, the continuum and the second uncountable cardinal}. 
%Proc. Amer. Math. Soc. 130 (2002), no. 9, 2753--2759. 


%\bibitem{Moore: weak D OCA}
%J. T. Moore.
%{\em Weak diamond and open colorings}. 
%J. Math. Log. 3 (2003), no. 1, 119--125.

%\bibitem{Moore: PFA}
%J. T. Moore.
%{\em The proper forcing axiom}. 
%Proceedings of the International Congress of Mathematicians. 
%Volume II, 3--29, Hindustan Book Agency, New Delhi, 2010.


%\bibitem{MHD: PDP}
%J. T. Moore, M. Hru\v s\'ak and M. D\v zamonja.
%{\em Parametrized $\diamondsuit$ principles}.
%Transactions of American Mathematical Society,
%356 (2004), no. 6, 2281--2306.

%\bibitem{Moore: MRP}
%J. T. Moore.
%{\em A five element basis for the uncountable linear orders}. 
%Ann. of Math. (2) 163 (2006), no. 2, 669--688.
%
%\bibitem{Moore: mho}
%J. T. Moore.
%{\em Aronszajn lines and the club filter}. 
%J. Symbolic Logic 73 (2008), no. 3, 1029--1035.
%
%\bibitem{Moore: MRP tutorial}
%J. T. Moore and D. Milovich.
%{\em A tutorial on Set Mapping Reflection}, 
%in Appalachian set theory 2006-2012, 
%J. Cummings and E. Schimmerling, eds. 
%London Math Society Lecture Notes series, 
%Cambridge University Press (2013).


%\bibitem{Roitman: S and L}
%J. Roitman.
%{\em Basic S and L}. 
%Handbook of set-theoretic topology, 295--326, 
%North-Holland, Amsterdam, 1984.

%\bibitem{Rowen}
%L. H. Rowen.
%Ring Theory. Student edition.
%{\em Academic Press, Inc., Boston, MA}, 1991.

\bibitem{Shelah:W}
S. Shelah.
Infinite abelian groups, Whitehead problem and some constructions. 
{\em Israel J. Math}. 18 (1974), 243--256. 

%\bibitem{Shelah: proper}
%S. Shelah.
%{\em Proper and improper forcing}.
%Second edition. Perspectives in Mathematical Logic. 
%Springer-Verlag, Berlin, 1998.

%\bibitem{Scheepers: gaps}
%M. Scheepers.
%{\em Gaps in $\omega^\omega$},
%In {\em Set Theory of the Reals},
%volume 6 of
%{\em Israel Mathematical Conference Proceedings},
%439--561, 1993.
%
\bibitem{SolovayTennenbaum: iteration}
R. Solovay and S. Tennenbaum.
{\em Iterated Cohen extensions and Souslin's problem}. 
Ann. of Math. (2) 94 (1971), 201--245.

%\bibitem{Steprans: ad paths}
%J. Stepr\= ans.
%{\em Almost disjoint families of paths in lattice grids}. 
%Topology Proc. 16 (1991), 185--200. 

%\bibitem{Todorcevic: note PFA}
%S. Todor\v cevi\'c.
%{\em A note on the proper forcing axiom}. 
%Axiomatic set theory (Boulder, Colo., 1983), 209--218, 
%Contemp. Math., 31, Amer. Math. Soc., Providence, RI, 1984.



%\bibitem{Todorcevic: directed set}
%S. Todor\v cevi\'c.
%{\em Directed sets and cofinal types},
%Transactions of American Mathematical Society,
%vol. 290, no. 2, pp. 711--723, 1985.
%
%
%\bibitem{Todorcevic:partitionproblems}
%S. Todor\v cevi\'c.
%{\em Partition Problems in Topology}.
%volume 84 of {\em Contemporary mathematics}.
%American Mathematical Society, Providence,
%Rhode Island, 1989.

%\bibitem{Todorcevic: gaps analytic}
%S. Todor\v cdvi\'c
%{\em Gaps in analytic quotients}. 
%Fund. Math. 156 (1998), no. 1, 85--97.


%\bibitem{Todorcevic: Basis problem}
%S. Todor\v cevi\'c.
%{\em Basis problems in combinatorial set theory}. 
%Proceedings of the International Congress of Mathematicians, 
%Vol. II (Berlin, 1998). Doc. Math. 1998, Extra Vol. II, 43--52.


%\bibitem{Todorcevic: PID}
%S. Todor\v cevi\'c.
%{\em A dichotomy for P-ideals of countable sets}. 
%Fund. Math. 166 (2000), no. 3, 251--267. 

%\bibitem{Todorcevic: ST gaps}
%S. Todor\v cevi\'c.
%{\em S-gaps and T-gaps}.
%Handwritten note, August 6, 2005.


%\bibitem{Todorcevic: Nogura}
%S. Todor\v cevi\'c.
%{\em A proof of Nogura's conjecture}. 
%Proc. Amer. Math. Soc. 131 (2003), no. 12, 3919--3923.

%\bibitem{Todorcevic: combinatorial}
%S. Todor\v cevi\'c.
%{\em Combinatorial dichotomies in set theory}, 
%Bull. Symbolic Logic 17 (2011), no. 1, 1--72.

%\bibitem{Todorcevic: PFA(S)[S]}
%S. Todor\v cevi\'c.
%{\em Forcing with a coherent Suslin tree}, 
%preprint. 


%\bibitem{TodorcevicFarah: applications}
%S. Todor\v cevi\'c and I. Farah.
%{\em Some applications of the method of forcing.}
%Yenisei Series in Pure and Applied Mathematics. 
%Yenisei, Moscow; Lyc\'ee, Troitsk, 1995.
%
%\bibitem{TodorcevicVelickovic: MA}
%S. Todor\v cevi\'c and B. Veli\v ckovi\'c.
%{\em Martin's axiom and partitions}. 
%Compositio Math. 63 (1987), no. 3, 391--408.

\bibitem{Tr ES}
J. Trlifaj.
Non-perfect rings and a theorem of Eklof and Shelah. 
{\it Comment. Math. Univ. Carolin}. 32 (1991), no. 1, 27--32. 

%\bibitem{Velickovic: applications OCA}
%B. Velickovic.
%{\em Applications of the open coloring axiom}. 
%Set theory of the continuum (Berkeley, CA, 1989), 137--154, 
%Math. Sci. Res. Inst. Publ., 26, Springer, New York, 1992.

%\bibitem{Velickovic: autom}
%B. Veli\v ckovi\'c.
%{\em ${\rm OCA}$ and automorphisms of
%$\cP(\omega)/{\rm fin}$}, 
%Topology and its Applications,
%vol. 49, no. 1, pp.1-13, 1993.

%\bibitem{Xu:ARC1}
%D. Xu.
%A note on the Auslander-Reiten conjecture. 
%{\em Acta Math. Sin. (Engl. Ser.)} 29 (2013), no. 10, 1993--1996.

%\bibitem{Xu:ARC2}
%D. Xu.
%Auslander-Reiten conjecture and special biserial algebras. 
%{\em Arch. Math. (Basel)} 105 (2015), no. 1, 13--22.

%\bibitem{Yorioka: D(non(M))=>dgap}
%T. Yorioka.
%{\it The diamond principle for the uniformity 
%of the meager ideal implies the existence 
%of a destructible gap}.
%Arch. Math. Logic 44 (2005), 677--683.
%
%\bibitem{Yorioka: cp ci dgap}
%T. Yorioka
%{\em Combinatorial principles on $\omega_1$, 
%cardinal invariants of the meager ideal and destructible gaps}. 
%J. Math. Soc. Japan 57 (2005), no. 4, 1217--1228.
%
%\bibitem{Yorioka: rec}
%T. Yorioka. 
%{\em Some weak fragments of Martin's axiom 
%related to the rectangle refining property}. 
%Arch. Math. Logic 47 (2008), no. 1, 79--90. 
%
%\bibitem{Yorioka:R1aleph1}
%T. Yorioka.
%{\em A non-implication between fragments of Martin's Axiom 
%related to a property which comes from Aronszajn trees}, 
%Ann. Pure Appl. Logic, 
%161 (2010), no. 4, 469--487. 
%
%\bibitem{Yorioka:R1aleph1correction}
%T. Yorioka.
%{\em A correction to 
%``A non-implication between fragments of Martin's Axiom 
%related to a property which comes from Aronszajn trees''}, 
%Ann. Pure Appl. Logic 162 (2011), no. 9, 752--754.

%\bibitem{Yorioka:MArec}
%T. Yorioka.
%Uniformizing ladder system colorings and the rectangle refining property
%{\em Proc. Amer. Math. Soc}. 138 (2010), no. 8, 2961--2971.


%
%
%\bibitem{Yorioka:AMMA}
%T. Yorioka.
%{\em Some consequences from Proper Forcing Axiom together with large continuum and the negation of Martin's Axiom}, Journal of the Mathematical Society of Japan, to appear. 
%
%
%\bibitem{Yorioka:nonspecialAtree}
%T. Yorioka.
%{\em The existence of a non-special Aronszajn tree and 
%some consequences of the Proper Forcing Axiom }, 
%submitted.  

%\bibitem{Zapletal: keeping}
%J. Zapletal. 
%{\em Keeping additivity of the null ideal small}. 
%Proc. Amer. Math. Soc. 125 (1997), no. 8, 2443--2451.


\end{thebibliography}
\end{document}